\documentclass[11pt, oneside]{amsart}   	
\usepackage{geometry}                		
\usepackage[parfill]{parskip}    		
\usepackage{graphicx}
\usepackage{color}
\usepackage{amssymb,amsmath}
\newtheorem{theorem}{Theorem}
\newtheorem{lemma}[theorem]{Lemma}

\theoremstyle{definition}
\newtheorem{remark}[theorem]{Remark}
\newtheorem{example}[theorem]{Example}
\newtheorem{problem}[theorem]{Problem}
\newtheorem{definition}[theorem]{Definition}
\usepackage{epstopdf}
\numberwithin{equation}{section}
\numberwithin{theorem}{section}

\author{Bent Fuglede}
\address{Department of Mathematical Sciences, University of Copenhagen,
2100 Copenhagen, Denmark}
\email{fuglede@math.ku.dk}

\author{Natalia Zorii}
\address{Institute of Mathematics,
National Academy of Sciences of Ukraine, Tereshchenkivska 3, 01601,
Kyiv-4, Ukraine}
\email{natalia.zorii@gmail.com}

\begin{document}

\title[An alternative concept of Riesz energy of measures]{An alternative concept of Riesz energy of measures with application to generalized condensers}

\begin{abstract}
In view of a recent example of a positive Radon measure $\mu$ on a domain $D\subset\mathbb R^n$, $n\geqslant3$, such that $\mu$ is of finite energy $E_g(\mu)$ relative to the $\alpha$-Green kernel $g$ on $D$, though the energy of $\mu-\mu^{D^c}$ relative to the $\alpha$-Riesz kernel $|x-y|^{\alpha-n}$, $0<\alpha\leqslant2$, is not well defined (here $\mu^{D^c}$ is the $\alpha$-Riesz swept measure of $\mu$ onto $D^c=\mathbb R^n\setminus D$), we propose a weaker concept of $\alpha$-Riesz energy for which this defect has been removed. This concept is applied to the study of a minimum weak $\alpha$-Riesz energy problem over (signed) Radon measures on $\mathbb R^n$ associated with a (generalized) condenser ${\mathbf A}=(A_1,D^c)$, where $A_1$ is a relatively closed subset of $D$. A solution to this problem exists if and only if the $g$-cap\-acity of $A_1$ is finite, which in turn holds if and only if there exists a so-called measure of the condenser $\mathbf A$, whose existence was analyzed earlier in different settings by Beurling, Deny, Kishi, Bliedtner, and Berg. Our analysis is based particularly on our recent result on the completeness of the cone of all positive Radon measures $\mu$ on $D$ with finite $E_g(\mu)$ in the metric determined by the norm $\|\mu\|_g:=\sqrt{E_g(\mu)}$. We also show that the pre-Hilbert space of Radon measures on $\mathbb R^n$ with finite weak $\alpha$-Riesz energy is isometrically imbedded into its completion, the Hilbert space of real-valued tempered distributions with finite energy, defined with the aid of Fourier transformation. This gives an answer in the negative to a question raised by Deny in 1950.
\end{abstract}
\maketitle

\section{Introduction}\label{sec1}
Let $D$ be an (open, connected) domain in $\mathbb R^n$, $n\geqslant3$. An ordered pair ${\mathbf A}=(A_1,D^c)$, where $A_1$ is a relatively closed subset of $D$ and $D^c:=\mathbb R^n\setminus D$, is termed a ({\it generalized\/}) {\it condenser\/} in $\mathbb R^n$, and $A_1$ and $D^c$ its {\it plates\/}. In a recent paper \cite{DFHSZ}, as a preparation for a study of minimum (standard) $\alpha$-Riesz energy problems over (signed) Radon measures on $\mathbb R^n$ associated with the condenser $\mathbf A$, one considered the $\alpha$-Green kernel $g=g_D^\alpha$ on $D$, associated with the $\alpha$-Riesz kernel $\kappa_\alpha(x,y):=|x-y|^{\alpha-n}$ of order $0<\alpha\leqslant2$ on $\mathbb R^n$ (see e.g.\ \cite[Chapter~IV, Section~5]{L}). It is stated in \cite[Lemma~2.4]{DFHSZ} that if a bounded positive Radon measure $\mu$ on $D$ has finite $\alpha$-Green energy
$E_g(\mu):=\iint g(x,y)\,d\mu(x)\,d\mu(y)$, then the (signed) Radon measure $\nu:=\mu-\mu^{D^c}$ on $\mathbb R^n$, $\mu^{D^c}$ being the $\alpha$-Riesz swept measure of $\mu$ onto $D^c$, must have finite $\alpha$-Riesz energy $E_\alpha(\nu):=\iint\kappa_\alpha(x,y)\,d\nu(x)\,d\nu(y)$
in the standard sense in which it is required that $E_\alpha(|\nu|)<\infty$, where $|\nu|:=\nu^++\nu^-$. Regrettably, the short proof of Lemma~2.4 in \cite{DFHSZ} was incomplete, and actually the lemma fails in general, as seen by the counterexample given in \cite[Appendix]{DFHSZ2}. To be precise, the quoted example shows that there exists a bounded positive Radon measure $\mu$ on $D$ with finite $E_g(\mu)$ such that $E_\alpha(\mu-\mu^{D^c})$ is not well defined.
In the present paper we argue that this failure may be viewed as an indication that the above standard notion of (finite) energy of signed measures is too restrictive when dealing with condenser problems.

We show in Theorem~\ref{thm} below that the quoted lemma does hold if we replace the standard concept of $\alpha$-Riesz energy $E_\alpha(\nu)$ of a (signed) Radon measure $\nu$ on $\mathbb R^n$ by a weaker concept, denoted ${\dot E}_\alpha(\nu)$ and defined essentially (see Definition~\ref{def-weak}) by
\begin{equation}\label{def-weakkk}{\dot E}_\alpha(\nu)=\int(\kappa_{\alpha/2}\nu)^2\,dm,\end{equation}
where it is required that $\kappa_{\alpha/2}\nu\in L^2(m)$, $\kappa_{\alpha/2}\nu(x):=\int\kappa_{\alpha/2}(x,y)\,d\nu(y)$ being the $\kappa_{\alpha/2}$-pot\-en\-tial of $\nu$ at $x\in\mathbb R^n$. (Throughout the paper $m$ denotes the $n$-dimensional Lebesgue measure on $\mathbb R^n$.) This definition of weak energy shall be seen in the light of the Riesz composition identity \cite{R}, cf.\ \cite[Eq.~1.1.3]{L}.

Denoting by $\mathcal E_\alpha(\mathbb R^n)$, resp.\ $\dot{\mathcal E}_\alpha(\mathbb R^n)$, the pre-Hilbert space of Radon measures $\nu$ on $\mathbb R^n$ with finite standard, resp.\ weak, $\alpha$-Riesz energy, we show in Theorem~\ref{dense} that $\mathcal E_\alpha(\mathbb R^n)$ is dense in $\dot{\mathcal E}_\alpha(\mathbb R^n)$ in the topology determined by the weak energy norm $\|\nu\|^{\cdot}_\alpha:=\sqrt{{\dot E}_\alpha(\nu)}$, as well as in the (induced) vague topology. This enables us to prove in Theorem~\ref{S-D} that $\dot{\mathcal E}_\alpha(\mathbb R^n)$ is isometrically imbedded into its completion, the Hilbert space $S_\alpha^*$ of real-valued tempered distributions $T\in S^*$ on $\mathbb R^n$ \cite{S} with finite energy, defined with the aid of the Fourier transform of $T$. Note that this result for $\mathcal E_\alpha(\mathbb R^n)$ in place of $\dot{\mathcal E}_\alpha(\mathbb R^n)$ goes back to Deny \cite{De1}. It is however still unknown whether there is a measure in $S_\alpha^*$ of noncompact support such that its weak energy is infinite (see Theorem~\ref{S-D-comp} and Section~\ref{open} below). A similar question for positive measures was raised by Deny \cite[p.~85]{De2}.

Combining Theorems~\ref{S-D} and~\ref{thm} with the counterexample in \cite[Appendix]{DFHSZ2} implies that there exists a linear combination of positive measures with infinite standard $\alpha$-Riesz energy and even with compact support, whose energy in $S_\alpha^*$ is finite. This gives an answer in the negative to the question raised by Deny in \cite[p.~125, Remarque]{De1}.

Based on Theorem~\ref{thm}, we next show that if $A_1$ is a relatively closed subset of $D$, then the finiteness of the $\alpha$-Green capacity $c_g(A_1)$ of $A_1$ is equivalent to
the existence of a (unique) solution $\dot{\lambda}_{{\mathbf A},\alpha}$ to the problem of minimizing ${\dot E}_\alpha(\nu)$ over a proper class of (signed) Radon measures $\nu$ on $\mathbb R^n$ with finite weak $\alpha$-Riesz energy, associated with the generalized condenser ${\mathbf A}:=(A_1,D^c)$. Furthermore, either of these two assertions is equivalent to
the existence of a so-called {\it condenser measure\/} $\mu_{{\mathbf A},\alpha}$ (see Definition~\ref{def-m-c}), analyzed earlier in different settings by Beurling and Deny \cite{D3} (see also \cite[Theorem~6.5]{L}), Kishi \cite{Ki}, Bliedtner \cite{Bl}, and Berg \cite{Berg}. Under the stated condition $c_g(A_1)<\infty$, it is shown that $\dot{\lambda}_{{\mathbf A},\alpha}$ and $\mu_{{\mathbf A},\alpha}$ are identical up to a normalizing factor, and they are related in the expected way to the $g^\alpha_D$-equilibrium measure on $A_1$. Finally, these results are specified for a standard condenser $\mathbf A$ with nonzero Euclidean distance between $A_1$ and $D^c$. See Theorems~\ref{th-ex}, \ref{th-st} and Section~\ref{rem-st}.

Our analysis is based particularly on our recent result on the perfectness of the $\alpha$-Green kernel $g$, which amounts to the completeness of the cone of all positive Radon measures $\mu$ on $D$ with finite $E_g(\mu)$ in the metric determined by the norm $\|\mu\|_g:=\sqrt{E_g(\mu)}$ \cite[Theorem~4.11]{FZ}.

\section{Preliminaries}\label{sec:princ}

Let $X$ be a locally compact (Hausdorff) space \cite[Chapter~I, Section~9, n$^\circ$\,7]{B1}, to be specified below, and $\mathfrak M(X)$ the linear
space of all real-valued (signed) Radon measures $\mu$ on $X$, equipped with the {\it vague\/} topology, i.e.\ the topology of
pointwise convergence on the class $C_0(X)$ of all continuous functions\footnote{When speaking of a continuous numerical function we understand that the values are {\it finite\/} real numbers.} on $X$ with compact
support. We refer the reader to \cite{B2} for the theory of measures and integration on a locally compact space, to be used throughout the paper (see also \cite{F1} for a short survey). In all that follows the integrals are understood as {\it upper\/} integrals~\cite{B2}.

For the purposes of the present study it is enough to assume that $X$ is metrizable and {\it countable at infinity\/}, where the latter means that $X$ can be represented as a countable union of compact sets \cite[Chapter~I, Section~9, n$^\circ$\,9]{B1}. Then the vague topology on $\mathfrak M(X)$ satisfies the first axiom of countability \cite[Remark~2.5]{DFHSZ1}, and vague convergence is entirely determined by convergence of sequences. The vague topology on $\mathfrak M(X)$ is Hausdorff; hence, a vague limit of any sequence in $\mathfrak M(X)$ is {\it unique\/} (whenever it exists).

We denote by $\mu^+$ and $\mu^-$ the positive and the negative parts, respectively, in the Hahn--Jordan decomposition of a measure $\mu\in\mathfrak M(X)$, and by $S^\mu_{X}=S(\mu)$ its support. A measure $\mu\in\mathfrak M(X)$ is said to be {\it bounded\/} if $|\mu|(X)<\infty$, where $|\mu|:=\mu^++\mu^-$. Let $\mathfrak M^+(X)$ stand for the (convex, vaguely closed) cone of all positive $\mu\in\mathfrak M(X)$.

We define a (function) {\it kernel\/} $\kappa(x,y)$ on $X$ as a positive, symmetric, lower semicontinuous (l.s.c.) function on $X\times X$. Given $\mu,\mu_1\in\mathfrak M(X)$, we denote by
$E_\kappa(\mu,\mu_1)$ and $\kappa\mu$, respectively, the (standard) {\it mutual
energy\/} and the {\it potential\/} relative to the kernel $\kappa$,
i.e.\footnote{When introducing notation about numerical quantities we assume
the corresponding object on the right to be well defined~--- as a finite real number or~$\pm\infty$.}
\begin{align*}
E_\kappa(\mu,\mu_1)&:=\iint\kappa(x,y)\,d\mu(x)\,d\mu_1(y),\\
\kappa\mu(x)&:=\int\kappa(x,y)\,d\mu(y),\quad x\in X.
\end{align*}
Note that $\kappa\mu(x)$ is well defined provided that $\kappa\mu^+(x)$ or $\kappa\mu^-(x)$ is finite, and then $\kappa\mu(x)=\kappa\mu^+(x)-\kappa\mu^-(x)$. In particular, if $\mu\in\mathfrak M^+(X)$ then $\kappa\mu(x)$ is defined everywhere and represents a l.s.c.\ positive function on $X$.
Also observe that $E_\kappa(\mu,\mu_1)$ is well defined and equal to $E_\kappa(\mu_1,\mu)$ provided that
$E_\kappa(\mu^+,\mu_1^+)+E_\kappa(\mu^-,\mu_1^-)$ or $E_\kappa(\mu^+,\mu_1^-)+E_\kappa(\mu^-,\mu_1^+)$ is finite.
For $\mu=\mu_1$, $E_\kappa(\mu,\mu_1)$ becomes the (standard) {\it energy\/} $E_\kappa(\mu)$. Let $\mathcal E_\kappa(X)$ consist
of all $\mu\in\mathfrak M(X)$ whose (standard) energy $E_\kappa(\mu)$ is finite, which means that $E_\kappa(\mu^+)$, $E_\kappa(\mu^-)$ and $E_\kappa(\mu^+,\mu^-)$ are all finite, and let $\mathcal E^+_\kappa(X):=\mathcal E_\kappa(X)\cap\mathfrak M^+(X)$.

Given a set $Q\subset X$, let $\mathfrak M^+(Q;X)$ consist of all $\mu\in\mathfrak M^+(X)$ {\it concentrated on\/}
$Q$, which means that $X\setminus Q$ is locally $\mu$-negligible, or equivalently that $Q$ is $\mu$-meas\-ur\-able and $\mu=\mu|_Q$, where $\mu|_Q=1_Q\cdot\mu$ is the trace (restriction) of $\mu$ on $Q$ \cite[Chapter~V, Section~5, n$^\circ$\,2, Example]{B2}. (Here $1_Q$ denotes the indicator function of $Q$.) If $Q$ is closed then $\mu$ is concentrated on $Q$ if and only if it is supported by $Q$, i.e.\ $S(\mu)\subset Q$. It follows from the countability of $X$ at infinity that the concept of local $\mu$-neg\-lig\-ibility coincides with that of $\mu$-negligibility; and hence $\mu\in\mathfrak M^+(Q;X)$ if and only if $\mu^*(X\setminus Q)=0$, $\mu^*(\cdot)$ being the {\it outer measure\/} of a set.
Write $\mathcal E_\kappa^+(Q;X):=\mathcal E_\kappa(X)\cap\mathfrak M^+(Q;X)$, $\mathfrak M^+(Q,q;X):=\{\mu\in\mathfrak M^+(Q;X):\ \mu(Q)=q\}$ and
$\mathcal E_\kappa^+(Q,q;X):=\mathcal E_\kappa(X)\cap\mathfrak M^+(Q,q;X)$, where $q\in(0,\infty)$.

Among the variety of potential-theoretic principles investigated for example in the comprehensive work by Ohtsuka~\cite{O} (see also the references therein), in the present study we shall only need the following two:
 \begin{itemize}
 \item[$\bullet$] A kernel $\kappa$ is said to satisfy the {\it complete maximum principle} (introduced by Cartan and Deny \cite{CD}) if for any $\mu\in\mathcal E^+_\kappa(X)$ and $\nu\in\mathfrak M^+(X)$ such that $\kappa\mu\leqslant\kappa\nu+c$ $\mu$-a.e., where $c\geqslant0$ is a constant, the same inequality holds everywhere on $X$.
\item[$\bullet$] A kernel $\kappa$ is said to be {\it positive definite\/} if $E_\kappa(\mu)\geqslant0$ for every (signed) measure $\mu\in\mathfrak M(X)$ for which the energy is well defined; and $\kappa$ is said to be {\it strictly positive definite\/}, or to satisfy the {\it energy principle\/} if in addition $E_\kappa(\mu)>0$ except for $\mu=0$.
\end{itemize}

{\it Unless explicitly stated otherwise, in all that
follows we assume a kernel $\kappa$ to satisfy the energy principle\/}. Then $\mathcal E_\kappa(X)$ forms a pre-Hil\-bert space with the inner product $\langle\mu,\nu\rangle_\kappa:=E_\kappa(\mu,\mu_1)$ and the (standard) energy norm $\|\mu\|_\kappa:=\sqrt{E_\kappa(\mu)}$ (see \cite{F1}). The (Hausdorff) topology
on $\mathcal E_\kappa(X)$ defined by the norm $\|\cdot\|_\kappa$ is termed {\it strong\/}.

The ({\it inner\/}) {\it capacity\/} $c_\kappa(Q)$ of a set $Q\subset X$ relative to the kernel $\kappa$ is given by \begin{equation}\label{cap-def}c_\kappa(Q)^{-1}:=\inf_{\mu\in\mathcal
E_\kappa^+(Q,1;X)}\,\|\mu\|_\kappa^2=:w_\kappa(Q)\end{equation}
(see e.g.\ \cite{F1,O}). Then $0\leqslant c_\kappa(Q)\leqslant\infty$. (As usual, here and in the sequel the
infimum over the empty set is taken to be $+\infty$. We also set
$1\bigl/(+\infty)=0$ and $1\bigl/0=+\infty$.)

Because of the strict positive definiteness of the kernel $\kappa$, $c_\kappa(K)<\infty$ for every compact set $K\subset X$.
Furthermore, by \cite[p.~153, Eq.~2]{F1},
\begin{equation}\label{compact}c_\kappa(Q)=\sup\,c_\kappa(K)\quad(K\subset Q, \ K\text{\ compact}).\end{equation}

An assertion $\mathcal U(x)$ involving a variable point $x\in X$ is said to hold {\it $c_\kappa$-ne\-ar\-ly everywhere\/} ({\it $c_\kappa$-n.e.\/}) on $Q\subset X$ if $c_\kappa(N)=0$ where $N$ consists of all $x\in Q$ for which $\mathcal U(x)$ fails to hold.
We shall often use the fact that $c_\kappa(N)=0$ if and only if $\mu_*(N)=0$ for every $\mu\in\mathcal E_\kappa^+(X)$, $\mu_*(\cdot)$ being the {\it inner measure\/} of a set (see \cite[Lemma~2.3.1]{F1}).

As in \cite[p.\ 134]{L}, we call a measure $\mu\in\mathfrak M(X)$ {\it $c_\kappa$-absolutely continuous\/} if $\mu(K)=0$ for every compact set $K\subset X$ with $c_\kappa(K)=0$. It follows from (\ref{compact}) that for such a $\mu$, $|\mu|_*(Q)=0$ for every $Q\subset X$ with $c_\kappa(Q)=0$. Hence every $\mu\in\mathcal E_\kappa(X)$ is $c_\kappa$-ab\-sol\-utely continuous; but not conversely \cite[pp.~134--135]{L}.

\begin{definition}\label{def-perf}Following~\cite{F1}, we call a (strictly positive definite)
kernel $\kappa$ {\it perfect\/} if every strong Cauchy sequence in $\mathcal E_\kappa^+(X)$ converges strongly to any of its vague cluster points\footnote{It follows from Theorem~\ref{fu-complete} that for a perfect kernel such a vague cluster point exists and is unique.}.\end{definition}

\begin{remark}\label{rem:clas} On $X=\mathbb R^n$, $n\geqslant3$, the $\alpha$-Riesz kernel $\kappa_\alpha(x,y)=|x-y|^{\alpha-n}$, $\alpha\in(0,n)$, is strictly positive definite and moreover perfect \cite{De1,De2}; thus so is the Newtonian kernel $\kappa_2(x,y)=|x-y|^{2-n}$ \cite{Ca}. Recently it has been shown by the present authors that if $X=D$ where $D$ is an arbitrary open set in $\mathbb R^n$, $n\geqslant3$, and $g^\alpha_D$, $\alpha\in(0,2]$, is the $\alpha$-Green kernel on $D$ \cite[Chapter~IV, Section~5]{L}, then $\kappa=g^\alpha_D$ is likewise strictly positive definite and moreover perfect \cite[Theorems~4.9, 4.11]{FZ}.\end{remark}

\begin{theorem}[{\rm see \cite{F1}}]\label{fu-complete} If a kernel\/ $\kappa$ on a locally compact space\/ $X$ is perfect, then the cone\/ $\mathcal E_\kappa^+(X)$ is strongly complete and the strong topology on\/ $\mathcal E_\kappa^+(X)$ is finer than the\/ {\rm(}induced\/{\rm)} vague topology on\/ $\mathcal E_\kappa^+(X)$.\end{theorem}

\begin{remark}\label{remma}In contrast to Theorem~\ref{fu-complete}, for a perfect kernel $\kappa$ the whole pre-Hilbert space $\mathcal E_\kappa(X)$ is in general strongly {\it incomplete\/}, and this is the case even for the $\alpha$-Riesz kernel of order $\alpha\in(1,n)$ on $\mathbb R^n$, $n\geqslant 3$
(see \cite{Ca} and \cite[Theorem~1.19]{L}). Compare with \cite[Theorem~1]{ZUmzh} where the strong completeness has been established for the metric subspace of all ({\it signed\/}) $\nu\in\mathcal E_{\kappa_\alpha}(\mathbb R^n)$ such that $\nu^+$ and $\nu^-$ are supported by closed nonintersecting sets in $\mathbb R^n$, $n\geqslant3$. This result from \cite{ZUmzh} has been proved with the aid of Deny's theorem~\cite{De1} stating that $\mathcal E_{\kappa_\alpha}(\mathbb R^n)$ can be completed by making use of tempered distributions on $\mathbb R^n$ with finite $\alpha$-Riesz energy, defined in terms of Fourier transforms (compare with Remark~\ref{remark}).\end{remark}

\begin{remark}\label{remark} The concept of perfect kernel is an efficient tool in minimum energy problems
over classes of {\it positive\/} Radon measures with
finite energy. Indeed, if $Q\subset X$ is closed, $c_\kappa(Q)\in(0,+\infty)$, and $\kappa$ is perfect, then the problem (\ref{cap-def}) has a unique solution $\lambda_{Q,\kappa}$ \cite[Theorem~4.1]{F1}. We shall call such a $\lambda_{Q,\kappa}$ the ({\it inner\/}) {\it $\kappa$-capacitary measure\/} on~$Q$. Later the concept of perfectness has been shown to be efficient also in minimum energy problems over classes of ({\it signed\/}) measures associated with a {\it standard condenser\/} \cite{ZPot1,ZPot2} (see also Remark~\ref{r-3} below for a short survey). In contrast to \cite[Theorem~1]{ZUmzh}, the approach developed in \cite{ZPot1,ZPot2} substantially used the assumption of the boundedness of the kernel on the product of the oppositely charged plates of a condenser, which made it possible to extend Cartan's proof \cite{Ca} of the strong completeness of the cone $\mathcal E_{\kappa_2}^+(\mathbb R^n)$ of all {\it positive\/} measures on $\mathbb R^n$ with finite Newtonian energy to an arbitrary perfect kernel $\kappa$ on a locally compact space $X$ and suitable classes of ({\it signed\/}) measures $\mu\in\mathcal E_\kappa(X)$.\end{remark}

A set $Q\subset X$ is said to be {\it locally closed\/} in $X$ if for every $x\in Q$ there is a neighborhood $V$ of $x$ in $X$ such that $V\cap Q$ is a closed subset of the subspace $Q$ \cite[Chapter~I, Section~3, Definition~2]{B1}, or equivalently if $Q$ is the intersection of an open and a closed subset of $X$ \cite[Chapter~I, Section~3, Proposition~5]{B1}.
The latter implies that this $Q$ is universally measurable, and hence $\mathfrak M^+(Q;X)$ consists of all the restrictions $\mu|_Q$ where $\mu$ ranges over $\mathfrak M^+(X)$. On the other hand, by \cite[Chapter~I, Section~9, Proposition~13]{B1} a locally closed set $Q$ itself can be thought of as a locally compact subspace of $X$. Thus $\mathfrak M^+(Q;X)$ consists, in fact, of all those $\nu\in\mathfrak M^+(Q)$ for each of which there is $\widehat{\nu}\in\mathfrak M^+(X)$ with the property
\begin{equation}\label{extend}\widehat{\nu}(\varphi)=\int1_{Q}\varphi\,d\nu\text{ \ for every \ }\varphi\in C_0(X).\end{equation}
We say that such $\widehat{\nu}$ {\it extends\/} $\nu\in\mathfrak M^+(Q)$ by $0$ off $Q$ to all of $X$.
A sufficient condition for (\ref{extend}) to hold is that $\nu$ be bounded.

\section{$\alpha$-Riesz balayage and $\alpha$-Green kernel}\label{sec:RG} Fix $n\geqslant3$ and $\alpha\in(0,n)$. {\it Unless explicitly stated otherwise, in all that follows we assume that\/} $\alpha\leqslant2$. Fix also a domain $D\subset\mathbb R^n$ with $c_{\kappa_\alpha}(D^c)>0$, where $D^c:=\mathbb R^n\setminus D$, and assume that either $\kappa(x,y)=\kappa_\alpha(x,y):=|x-y|^{\alpha-n}$ is the {\it $\alpha$-Riesz kernel\/} on $X=\mathbb R^n$, or $\kappa(x,y)=g_D^\alpha(x,y)$ is the {\it $\alpha$-Green kernel\/} on $X=D$. For the definition of $g=g_D^\alpha$, see \cite[Chapter~IV, Section~5]{L} or see below. We shall simply write $\alpha$ instead of $\kappa_\alpha$ if $\kappa_\alpha$ serves as an index.

When speaking of a positive Radon measure $\mu\in\mathfrak M^+(\mathbb R^n)$, we always tacitly assume that for the given $\alpha$, $\kappa_\alpha\mu\not\equiv+\infty$. This implies that
\begin{equation}\label{1.3.10}\int_{|y|>1}\,\frac{d\mu(y)}{|y|^{n-\alpha}}<\infty\end{equation}
(see \cite[Eq.~1.3.10]{L}), and consequently that $\kappa_\alpha\mu$ is finite $c_\alpha$-n.e.\ on $\mathbb R^n$ \cite[Chap\-ter~III, Section~1]{L}; these two implications can actually be reversed.

We shall often use the short form 'n.e.' instead of '$c_\alpha$-n.e.' if this will not cause any mis\-under\-standing.

\begin{definition}\label{d-ext} $\nu\in\mathfrak M(D)$ is called {\it extendible\/} if there exist $\widehat{\nu^+}$ and $\widehat{\nu^-}$ extending $\nu^+$ and $\nu^-$, respectively, by~$0$ off $D$ to $\mathbb R^n$ (see (\ref{extend})), and if these $\widehat{\nu^+}$ and $\widehat{\nu^-}$ satisfy (\ref{1.3.10}).
We identify such a $\nu\in\mathfrak M(D)$ with its extension $\widehat\nu:=\widehat{\nu^+}-\widehat{\nu^-}$, and we therefore write $\widehat\nu=\nu$.\end{definition}

Every bounded measure $\nu\in\mathfrak M(D)$ is extendible. The converse holds if $D$ is bounded, but not in general (e.g.\ not if $D^c$ is compact). The set of all extendible measures consists of all the restrictions $\mu|_D$ where $\mu$ ranges over $\mathfrak M(\mathbb R^n)$. Also observe that for any extendible measure $\nu\in\mathfrak M(D)$, $\kappa_{\alpha}\nu$ is well defined and finite n.e.\ on $\mathbb R^n$, for $\kappa_\alpha\nu^\pm$ is so.

The {\it $\alpha$-Green kernel\/} $g=g_D^\alpha$ on $D$ is defined by
\begin{equation*}g^\alpha_D(x,y)=\kappa_\alpha\varepsilon_y(x)-
\kappa_\alpha\varepsilon_y^{D^c}(x)\text{ \ for all \ }x,y\in D,\end{equation*}
where $\varepsilon_y$ denotes the unit Dirac measure at a
point $y$ and $\varepsilon_y^{D^c}$ its {\it $\alpha$-Riesz balayage\/} (sweeping) onto the (closed) set $D^c$, uniquely determined in the frame of the classical approach by \cite[Theorem~3.6]{FZ} pertaining to positive Radon measures on $\mathbb R^n$. See also the book by Bliedtner and Hansen \cite{BH} where balayage is studied in the setting of balayage spaces.

We shall simply write $\mu'$ instead of $\mu^{D^c}$ when speaking of the $\alpha$-Riesz balayage of $\mu\in\mathfrak M^+(D;\mathbb R^n)$ onto $D^c$. According to \cite[Corollaries~3.19, 3.20]{FZ}, for any $\mu\in\mathfrak M^+(D;\mathbb R^n)$ {\it the balayage\/ $\mu'$ is\/ $c_\alpha$-absolutely continuous and it is determined uniquely by relation
\begin{equation}\label{bal-eq}\kappa_\alpha\mu'=\kappa_\alpha\mu\text{ \ n.e.\ on \ }D^c\end{equation}
among the\/ $c_\alpha$-absolutely continuous measures supported by\/} $D^c$. If moreover\/ $\mu\in\mathcal E_\alpha^+(D;\mathbb R^n)$, then the balayage $\mu'$ is in fact
{\it the orthogonal projection of\/ $\mu$ on the convex cone\/ $\mathcal E^+_\alpha(D^c;\mathbb R^n)$} \cite[Theorem~3.1]{FZ}, i.e.\ $\mu'\in\mathcal E^+_\alpha(D^c;\mathbb R^n)$ and
\begin{equation}\label{proj}\|\mu-\theta\|_\alpha>\|\mu-\mu'\|_\alpha\text{ \ for all \ }\theta\in\mathcal E^+_\alpha(D^c;\mathbb R^n),\quad\theta\ne\mu'.\end{equation}

If now $\nu\in\mathfrak M(D)$ is an extendible (signed) measure, then $\nu':=\nu^{D^c}:=(\nu^+)^{D^c}-(\nu^-)^{D^c}$
is said to be a {\it balayage\/} of $\nu$ onto $D^c$. It follows from \cite[p.~178, Remark]{L} that the balayage $\nu'$ is determined uniquely by (\ref{bal-eq}) with $\nu$ in place of $\mu$ among the $c_\alpha$-absolutely continuous (signed) measures supported by~$D^c$.

\begin{definition}[{\rm see \cite[Theorem~VII.13]{Brelo2}}]A closed set $F\subset\mathbb R^n$ is said to be {\it $\alpha$-thin
at infinity\/} if either $F$ is compact, or the inverse of $F$ relative to $\{x\in\mathbb R^n: |x|=1\}$ has $x=0$ as an $\alpha$-irregular boundary point (cf.\ \cite[Theorem~5.10]{L}).\end{definition}

\begin{remark}\label{rem-thin}Any closed set $F$ that is not $\alpha$-thin at infinity is of infinite capacity $c_\alpha(F)$. Indeed, by the Wiener criterion of $\alpha$-regularity, $F$ is not $\alpha$-thin at infinity if and only if
\[\sum_{k\in\mathbb N}\,\frac{c_\alpha(F_k)}{q^{k(n-\alpha)}}=\infty,\]
where $q>1$ and $F_k:=F\cap\{x\in\mathbb R^n: q^k\leqslant|x|<q^{k+1}\}$, while by \cite[Lemma~5.5]{L} $c_\alpha(F)<\infty$ is equivalent to the relation
\[\sum_{k\in\mathbb N}\,c_\alpha(F_k)<\infty.\]
These observations also imply that the converse is not true, i.e.\ there is $F$ with $c_\alpha(F)=\infty$, but $\alpha$-thin at infinity (see also \cite[pp.~276--277]{Ca2}).\end{remark}

\begin{example}[{\rm see \cite[Example~5.3]{ZPot1}}]\label{ex-thin}Let $n=3$ and $\alpha=2$. Define the rotation body
\begin{equation}\label{descr}F_\varrho:=\bigl\{x\in\mathbb R^3: \ 0\leqslant x_1<\infty, \ x_2^2+x_3^2\leqslant\varrho^2(x_1)\bigr\},\end{equation}
where $\varrho$ is given by one of the following three formulae:
\begin{align}
\label{c1}\varrho(x_1)&=x_1^{-s}\text{ \ with \ }s\in[0,\infty),\\
\label{c2}\varrho(x_1)&=\exp(-x_1^s)\text{ \ with \ }s\in(0,1],\\
\label{c3}\varrho(x_1)&=\exp(-x_1^s)\text{ \ with \ }s\in(1,\infty).
\end{align}
Then $F_\varrho$ is not $2$-thin at infinity if $\varrho$ is defined by (\ref{c1}), $F_\varrho$ is $2$-thin at infinity but has infinite Newtonian capacity if $\varrho$ is given  by (\ref{c2}), and finally $c_2(F_\varrho)<\infty$ if (\ref{c3}) holds.
\end{example}

\begin{theorem}[{\rm see \cite[Theorem~3.22]{FZ}}]\label{bal-mass-th} The set\/ $D^c$ is not\/ $\alpha$-thin at infinity if and only if for every bounded measure\/ $\mu\in\mathfrak M^+(D)$ we have $\mu'(\mathbb R^n)=\mu(\mathbb R^n)$.\footnote{In general, $\nu^{D^c}(\mathbb R^n)\leqslant\nu(\mathbb R^n)$ for every $\nu\in\mathfrak M^+(\mathbb R^n)$ \cite[Theorem~3.11]{FZ}.\label{foot-bound}}
\end{theorem}

As noted in Remark~\ref{rem:clas} above, the $\alpha$-Riesz kernel $\kappa_\alpha$ on $\mathbb R^n$ and the $\alpha$-Green kernel $g=g^\alpha_D$ on $D$ are both strictly positive definite and moreover perfect. Furthermore, the $\alpha$-Riesz kernel $\kappa_\alpha$ (with $\alpha\in(0,2]$) satisfies the complete maximum principle in the form stated in Section~\ref{sec:princ} (see \cite[Theorems~1.27, 1.29]{L}). Regarding a similar result for the $\alpha$-Green kernel $g$, the following assertion holds.

\begin{theorem}[{\rm see \cite[Theorem~4.6]{FZ}}]\label{th-dom-pr} Let\/ $\mu\in\mathcal E^+_g(D)$, let\/ $\nu\in\mathfrak M^+(D)$ be extendible, and let\/ $w$ be a positive\/ $\alpha$-super\-har\-monic function on\/ $\mathbb R^n$ {\rm\cite[{\it Chapter\/}~I, {\it Section\/}~5, {\it n\/}$^\circ$\,20]{L}}. If moreover\/ $g\mu\leqslant g\nu+w$ $\mu$-a.e.\ on\/ $D$, then the same inequality holds on all of\/~$D$.\end{theorem}

The following four lemmas establish relations between potentials and energies relative to the kernels $\kappa_\alpha$ and $g=g^\alpha_D$, respectively.

\begin{lemma}[{\rm see \cite[Lemma~3.4]{DFHSZ2}}]\label{l-hatg} For any extendible measure\/ $\mu\in\mathfrak M(D)$ the $\alpha$-Green potential\/ $g\mu$ is well defined and finite\/ {\rm(}$c_\alpha$-{\rm)}n.e.\ on\/ $D$ and given by\footnote{If $N$ is a given subset of $D$, then $c_g(N)=0$ if and only if $c_\alpha(N)=0$ \cite[Lemma~2.6]{DFHSZ}. Thus any assertion involving a variable point holds n.e.\ on $Q\subset D$ if and only if it holds $c_g$-n.e.\ on $Q$.\label{RG}}
\begin{equation}\label{hatg}g\mu=\kappa_\alpha\mu-\kappa_\alpha\mu'\text{ \ n.e.\ on \ }D.\end{equation}
\end{lemma}

\begin{lemma}[{\rm see \cite[Lemma~3.5]{DFHSZ2}}]\label{l-hen'} Suppose that\/ $\mu\in\mathfrak M(D)$ is extendible and the extension belongs to\/ $\mathcal E_\alpha(\mathbb R^n)$. Then\/ $\mu\in\mathcal E_g(D)$, $\mu-\mu'\in\mathcal E_\alpha(\mathbb R^n)$ and moreover\footnote{Compare with the faulty Lemma~2.4 from \cite{DFHSZ}, mentioned in the Introduction.}
\begin{equation}\label{gr}\|\mu\|^2_g=\|\mu-\mu'\|^2_\alpha=\|\mu\|^2_\alpha-\|\mu'\|^2_\alpha.\end{equation}
\end{lemma}

\begin{lemma}[{\rm see \cite[Lemma~3.6]{DFHSZ2}}]\label{l-hen'-comp}
Assume that\/ $\mu\in\mathfrak M(D)$ has compact support\/ $S^\mu_D$. Then\/ $\mu\in\mathcal E_g(D)$ if and only if its extension belongs to\/ $\mathcal E_\alpha(\mathbb R^n)$.\end{lemma}

Lemma~\ref{l-hen'-comp} can be generalized as follows.

\begin{lemma}\label{eq-r-g}Let\/ $A_1$ be a relatively closed subset of\/ $D$ such that
\begin{equation}\label{dist}{\rm dist}\,(A_1,D^c):=\inf_{x\in A_1, \ z\in D^c}\,|x-z|>0.\end{equation}
Then a bounded measure\/ $\mu\in\mathfrak M^+(A_1;D)$ has finite\/ $E_g(\mu)$ if and only if its extension has finite\/ $\alpha$-Riesz energy, and in the affirmative case relation\/ {\rm(\ref{gr})} holds. Furthermore,
$c_g(A_1)<\infty$ if and only if\/ $c_\alpha(A_1)<\infty$.\end{lemma}

\begin{proof}Since $\nu^{D^c}(\mathbb R^n)\leqslant\nu(\mathbb R^n)$ for any $\nu\in\mathfrak M^+(\mathbb R^n)$ \cite[Theorem~3.11]{FZ}, we get from~(\ref{dist})
\[\kappa_\alpha\varepsilon_y'(x)=\int|x-z|^{\alpha-n}\,d\varepsilon_y'(z)\leqslant{\rm dist}\,(A_1,D^c)^{\alpha-n}=:C<\infty\text{ \ for all \ }x,y\in A_1,\]
where the constant $C$ is independent of $x,y\in A_1$. Hence
\[\kappa_\alpha(x,y)=g(x,y)+\kappa_\alpha\varepsilon_y'(x)\leqslant g(x,y)+C\text{ \ for all \ }x,y\in A_1,\]
which in turn yields
\begin{equation}\label{est}E_\alpha(\mu)\leqslant E_g(\mu)+C\mu(D)^2\text{ \ for every bounded \ }\mu\in\mathfrak M^+(A_1;D).\end{equation}
Thus $E_g(\mu)<\infty$ implies $E_\alpha(\mu)<\infty$, which according to Lemma~\ref{l-hen'} establishes (\ref{gr}). Since $g(x,y)<\kappa_\alpha(x,y)$ for all $x,y\in D$,\footnote{The strict inequality in this relation is caused by our convention $c_\alpha(D^c)>0$.} the relation $E_\alpha(\mu)<\infty$ implies $E_g(\mu)<\infty$, and the first assertion of the lemma follows.

In particular, we obtain from (\ref{est})
\[E_\alpha(\mu)\leqslant E_g(\mu)+C\text{ \ for every \ }\mu\in\mathfrak M^+(A_1,1;D).\]
Hence, $c_g(A_1)$ is finite if $c_\alpha(A_1)$ is so. As the converse is obvious, the proof is complete.
\end{proof}

\section{A weaker notion of $\alpha$-Riesz energy}\label{weak}

Recall that according to our general convention stated at the beginning of Section~\ref{sec:RG} for the given $\alpha$, $0<\alpha\leqslant 2$, we have $\kappa_\alpha|\mu|\not\equiv\infty$ for any $\mu\in\mathfrak M(\mathbb R^n)$.

As seen from \cite[Lemma~3.1.1]{F1} and the above definition of energy of a (signed) measure, $\mu\in\mathcal E_\alpha(\mathbb R^n)$ if and only if $\mu^\pm\in\mathcal E^+_\alpha(\mathbb R^n)$. For our purposes this standard concept of finite energy is too restrictive (see the Introduction), and we therefore need to consider also the following weaker notion of finite energy, inspired by the Riesz composition identity \cite{R} (see also \cite[Eq.~1.1.3]{L} with both $\alpha$ and $\beta$ replaced by the present $\alpha/2$).

\begin{definition}\label{def-weak} A (signed) measure $\mu\in\mathfrak M(\mathbb R^n)$ is said to have {\it finite weak\/ $\alpha$-Riesz energy} if $\kappa_{\alpha/2}\mu\in L^2(m)$, where $m$ is the $n$-dim\-en\-sional Lebesgue measure on $\mathbb R^n$. In the affirmative case the weak $\alpha$-Riesz energy ${\dot E}_\alpha(\mu)$ of $\mu$ is given by~(\ref{def-weakkk}).
\end{definition}

Note that we do not require that $\kappa_{\alpha/2}|\mu|\in L^2(m)$, for that would mean that $E_\alpha(|\mu|)<\infty$, or equivalently that $E_\alpha(\mu)<\infty$, so the two notions of finite energy would be identical, and that is actually not the case according to Theorem~\ref{thm} below combined with the counterexample in \cite[Appendix]{DFHSZ2}. The class of all signed, resp.\ positive, measures of finite weak ($\alpha$-Riesz) energy is denoted by $\dot{\mathcal E}_\alpha(\mathbb R^n)$, resp.\ $\dot{\mathcal E}_\alpha^+(\mathbb R^n)$.

It follows from the Riesz composition identity and Fubini's theorem that $\mathcal E_\alpha(\mathbb R^n)\subset\dot{\mathcal E}_\alpha(\mathbb R^n)$ and
\begin{equation}\label{dot0}E_\alpha(\mu)={\dot E}_\alpha(\mu)\text{ \ for any \ }\mu\in{\mathcal E}_\alpha(\mathbb R^n)\end{equation}
(cf.\ \cite[Proof of Theorem~1.15]{L}). Furthermore, ${\mathcal E}_\alpha^+(\mathbb R^n)={\dot{\mathcal E}}_\alpha^+(\mathbb R^n)$. In other terms, every signed measure of standard finite energy has the same weak energy, and for positive measures the two concepts of energy are identical. A signed measure $\mu$ may however be of class $\dot{\mathcal E}_\alpha(\mathbb R^n)$ but not of class $\mathcal E_\alpha(\mathbb R^n)$, as noted above.

Clearly, ${\dot{\mathcal E}}_\alpha(\mathbb R^n)$ is a linear subspace of $\mathfrak M(\mathbb R^n)$ and a pre-Hilbert space with the (weak energy) norm $\|\mu\|^{\cdot}_\alpha:=\sqrt{{\dot E}_\alpha(\mu)}=\|\kappa_{\alpha/2}\mu\|_{L^2(m)}$ and the (weak) inner product
\[\langle\mu,\nu\rangle^\cdot_\alpha:=\langle\kappa_{\alpha/2}\mu,\kappa_{\alpha/2}\nu\rangle_{L^2(m)},\text{  \ where \ }\mu,\nu\in{\dot{\mathcal E}}_\alpha(\mathbb R^n).\] In fact, $\|\mu\|_\alpha^\cdot=0$ implies $\kappa_{\alpha/2}\mu=0$ $m$-a.e., and hence $\mu=0$ by \cite[Theorem~1.12]{L}. Moreover,
\begin{equation}\label{dot00}\langle\kappa_{\alpha/2}\mu,\kappa_{\alpha/2}\nu\rangle_{L^2(m)}=\int\kappa_\alpha\mu\,d\nu=\langle\mu,\nu\rangle_\alpha\text{ \ for \ }\mu,\nu\in\mathcal E_\alpha(\mathbb R^n).
\end{equation}
Indeed,
\[\|\kappa_{\alpha/2}(\mu+\nu)\|_{L^2(m)}^2=\|\kappa_{\alpha/2}\mu\|_{L^2(m)}^2+\|\kappa_{\alpha/2}\nu\|_{L^2(m)}^2+
2\langle\kappa_{\alpha/2}\mu,\kappa_{\alpha/2}\nu\rangle_{L^2(m)}.\]
Comparing this with $E_\alpha(\mu+\nu)=E_\alpha(\mu)+E_\alpha(\nu)+2\int\kappa_\alpha\mu\,d\nu$, we get~(\ref{dot00}) from (\ref{dot0}).

Let $\mathcal E_\alpha^b(\mathbb R^n)$ stand for the subspace of $\mathcal E_\alpha(\mathbb R^n)$ consisting of all {\it bounded\/} measures.

\begin{theorem}\label{dense} $\mathcal E_\alpha^b(\mathbb R^n)$ is dense in the pre-Hil\-bert space\/ $\dot{\mathcal E}_\alpha(\mathbb R^n)$ in the topology determined by the weak energy norm\/ $\|\cdot\|^{\cdot}_\alpha$, as well as in the\/ {\rm(}induced\/{\rm)} vague topology.
\end{theorem}

\begin{proof} Let $\eta$ be the $\kappa_\alpha$-capacitary measure on $\{x\in\mathbb R^n: |x|\leqslant1\}$ (see Remark~\ref{remark}). With any (signed) measure $\mu\in\dot{\mathcal E}_\alpha(\mathbb R^n)$ we associate two sequences $\{\mu^{\pm}_j\}_{j\in\mathbb N}$  characterized by
\[\kappa_{\alpha/2}\mu^{\pm}_j:=(j\kappa_{\alpha/2}\eta)\wedge\kappa_{\alpha/2}\mu^{\pm},\quad j\in\mathbb N,\]
and we next show that $\mu^{\pm}_j\in\mathcal E^b_\alpha(\mathbb R^n)$. According to \cite[Theorem~1.31]{L}, $(j\kappa_{\alpha/2}\eta)\wedge\kappa_{\alpha/2}\mu^{\pm}$ is indeed the $\alpha/2$-Riesz potential of a positive measure $\mu_j^\pm$, which is unique by \cite[Theorem~1.12]{L}. By the principle of positivity of mass \cite[Theorem~3.11]{FZ}, $\mu_j^\pm(\mathbb R^n)\leqslant j\eta(\mathbb R^n)$, and hence $\mu_j^\pm$ is bounded along with $j\eta$. Since $\kappa_{\alpha/2}\mu_j^\pm$ is majorized by $j\kappa_{\alpha/2}\eta$, they are both of the class $L^2(m)$ because $\dot E_\alpha(\eta)=E_\alpha(\eta)<\infty$. Finally, by the Riesz composition identity and Fubini's theorem,
\begin{align}\label{conv}\kappa_\alpha\mu_j^\pm&=\kappa_\alpha\ast\mu_j^\pm=
(\kappa_{\alpha/2}\ast\kappa_{\alpha/2})\ast\mu_j^\pm=\kappa_{\alpha/2}\ast(\kappa_{\alpha/2}\ast\mu_j^\pm)\\
\notag{}&\leqslant j\kappa_{\alpha/2}\ast(\kappa_{\alpha/2}\ast\eta)=j(\kappa_{\alpha/2}\ast\kappa_{\alpha/2})\ast\eta=j\kappa_\alpha\ast\eta=j\kappa_\alpha\eta,
\end{align}
which is finite on $\mathbb R^n$.
Here $\kappa_\alpha$ means the function $x\mapsto|x|^{\alpha-n}$ in the presence of convolution (and similarly with $\alpha/2$ in place of $\alpha$).
Altogether, $\mu_j^\pm\in\dot{\mathcal E}_\alpha^+(\mathbb R^n)={\mathcal E}_\alpha^+(\mathbb R^n)$.
Our next aim is to show that the (signed) measures $\mu_j:=\mu_j^+-\mu_j^-\in\mathcal E_\alpha^b(\mathbb R^n)$ (not a Hahn--Jordan decomposition) approach $\mu$ when $j\to\infty$ both in the norm $\|\cdot\|^\cdot_\alpha$ and vaguely.

Note that the two increasing sequences $\{\kappa_{\alpha/2}\mu_j^\pm\}_{j\in\mathbb N}$ converge pointwise to $\kappa_{\alpha/2}\mu^\pm\not\equiv\infty$.\footnote{For any $\nu\in\mathfrak M(\mathbb R^n)$ we have $\kappa_{\alpha/2}|\nu|\not\equiv\infty$. This follows from our general convention $\kappa_\alpha|\nu|\not\equiv\infty$ with the aid of the Riesz composition identity and Fubini's theorem (see the former line in (\ref{conv}) with $|\nu|$ in place of $\mu_j^\pm$). Thus for any $\nu\in\mathfrak M(\mathbb R^n)$, $\kappa_{\alpha/2}\nu$ is finite $c_{\alpha/2}$-n.e.\ on $\mathbb R^n$, for $\kappa_{\alpha/2}\nu^\pm$ is so.\label{foot1}}
Applying \cite[Theorem~3.9]{L} with $\alpha$ replaced throughout by $\alpha/2$, we thus see that $\mu^\pm_j\to\mu^\pm$ vaguely. Furthermore, $\kappa_{\alpha/2}\mu_j\to\kappa_{\alpha/2}\mu$ pointwise $c_{\alpha/2}$-n.e., hence $m$-a.e., for $m$ is $c_{\alpha/2}$-abs\-ol\-utely continuous. There is dominated  $L^2(m)$-con\-verg\-ence here since it will be shown that
\begin{equation}\label{(a)}|\kappa_{\alpha/2}\mu_j|\leqslant|\kappa_{\alpha/2}\mu|\text{ \ $c_{\alpha/2}$-n.e.\ (and hence $m$-a.e.) on $\mathbb R^n$}\end{equation}
and since $\kappa_{\alpha/2}\mu\in L^2(m)$ by Definition~\ref{def-weak}. This altogether will imply by Lebesgue's integration theorem \cite[Chapter~IV, Section~3, Theorem~6]{B2} that $\kappa_{\alpha/2}\mu_j\to\kappa_{\alpha/2}\mu$ also in $L^2(m)$-norm, that is, $\|\mu_j-\mu\|_\alpha^{\cdot}\to0$ as $j\to\infty$.

The proof is thus reduced to establishing (\ref{(a)}).
When considered first on the set $Q_1:=\{\kappa_{\alpha/2}\mu^+\geqslant\kappa_{\alpha/2}\mu^-\}\cap\{j\kappa_{\alpha/2}\eta\geqslant\kappa_{\alpha/2}\mu^-\}$, we proceed to show that
\begin{equation}\label{dot2}0\leqslant(j\kappa_{\alpha/2}\eta\wedge\kappa_{\alpha/2}\mu^+)-\kappa_{\alpha/2}\mu^-=\kappa_{\alpha/2}\mu_j
\leqslant\kappa_{\alpha/2}\mu^+-\kappa_{\alpha/2}\mu^-=\kappa_{\alpha/2}\mu
\end{equation}
wherever both $\kappa_{\alpha/2}\mu$ and $\kappa_{\alpha/2}\mu_j$ are well defined, i.e.\ $c_{\alpha/2}$-n.e. The former inequality and the latter equality in (\ref{dot2}) obviously both hold $c_{\alpha/2}$-n.e.\ on $Q_1$. The former equality and the latter inequality are obtained by inserting $(j\kappa_{\alpha/2}\eta)\wedge\kappa_{\alpha/2}\mu^+=\kappa_{\alpha/2}\mu_j^+\leqslant\kappa_{\alpha/2}\mu^+$ and  $\kappa_{\alpha/2}\mu^-_j=(j\kappa_{\alpha/2}\eta)\wedge\kappa_{\alpha/2}\mu^-=\kappa_{\alpha/2}\mu^-$. Thus (\ref{dot2}) and hence (\ref{(a)}) hold $c_{\alpha/2}$-n.e.\ on~$Q_1$.

On the next set $Q_2:=\{\kappa_{\alpha/2}\mu^+\geqslant\kappa_{\alpha/2}\mu^-\}\cap\{j\kappa_{\alpha/2}\eta<\kappa_{\alpha/2}\mu^-\}$ we have
$\kappa_{\alpha/2}\mu_j^-=(j\kappa_{\alpha/2}\eta)\wedge\kappa_{\alpha/2}\mu^-=j\kappa_{\alpha/2}\eta$ and furthermore $\kappa_{\alpha/2}\mu_j^+=(j\kappa_{\alpha/2}\eta)\wedge\kappa_{\alpha/2}\mu^+=j\kappa_{\alpha/2}\eta$. Altogether, $0=\kappa_{\alpha/2}\mu_j\leqslant\kappa_{\alpha/2}\mu$ $c_{\alpha/2}$-n.e.\ on~$Q_2$.

On the remaining set $Q_3:=\{\kappa_{\alpha/2}\mu^+<\kappa_{\alpha/2}\mu^-\}$ we have $\kappa_{\alpha/2}(-\mu)^+>\kappa_{\alpha/2}(-\mu)^-$, and the analysis from the preceding two paragraphs applies with $\mu$ replaced by $-\mu\in\dot{\mathcal E}_{\alpha}(\mathbb R^n)$. This leads to $0\geqslant\kappa_{\alpha/2}\mu_j\geqslant\kappa_{\alpha/2}\mu$ $c_{\alpha/2}$-n.e.\ on $Q_3$, and thus altogether to (\ref{(a)}), now verified $c_{\alpha/2}$-n.e.\ throughout $\mathbb R^n$.
\end{proof}

As was shown in Deny \cite{De1} (see also \cite[Chapter~I, Section~1]{L} for a brief survey), the pre-Hilbert space $\mathcal E_\alpha(\mathbb R^n)$ can be isometrically imbedded into its completion, the space $S_\alpha^*$ of real-valued tempered distributions $T\in S^*$ on $\mathbb R^n$ \cite{S} with finite Deny--Schwartz energy
\[\|T\|_{S_\alpha^*}^2=C_{n,\alpha}\int_{\mathbb
R^n}\frac{|\mathcal F[T](\xi)|^2}{|\xi|^\alpha}\,dm(\xi),\]
where $C_{n,\alpha}\in(0,\infty)$ depends on $n$ and $\alpha$ only, and $\mathcal F[T]$ is the Fourier transform of $T\in S^*$. It is shown below that the same holds for $\dot{\mathcal E}_\alpha(\mathbb R^n)$ in place of $\mathcal E_\alpha(\mathbb R^n)$.

\begin{theorem}\label{S-D} $\dot{\mathcal E}_\alpha(\mathbb R^n)$ is isometrically imbedded into its completion, the Hilbert space\/~$S_\alpha^*$.\end{theorem}

\begin{proof}As seen from what has been recalled just above, it is enough to show that
\begin{equation}\label{D-D}\mu\in S_\alpha^*\text{ \ and \ }\|\mu\|_\alpha^\cdot=\|\mu\|_{S_\alpha^*}\text{ \ for any $\mu\in\dot{\mathcal E}_\alpha(\mathbb R^n)$}.\end{equation}
Given $\mu\in\dot{\mathcal E}_\alpha(\mathbb R^n)$, choose $\mu_j\in\mathcal E_\alpha(\mathbb R^n)$, $j\in\mathbb N$, as in the proof of Theorem~\ref{dense}. Then $\{\mu_j\}_{j\in\mathbb N}$ is a strong Cauchy sequence in $\dot{\mathcal E}_\alpha(\mathbb R^n)$, hence by (\ref{dot0}) also in $\mathcal E_\alpha(\mathbb R^n)$, converging  to $\mu$ both in the weak energy norm $\|\cdot\|^\cdot_\alpha$ and vaguely.
According to \cite[Eq.~6.1.1]{L}, $\{\mu_j\}_{j\in\mathbb N}$ is also Cauchy in $S_\alpha^*$, and hence by \cite[Theorem~6.1]{L} it converges to some $T\in S_\alpha^*$ in the norm $\|\cdot\|_{S_\alpha^*}$. Thus by \cite[p.~120, Th\'eor\`eme~2]{De1}
\[T(\varphi)=\lim_{j\to\infty}\,\mu_j(\varphi)=\mu(\varphi)\text{ \ for every $\varphi\in C^\infty_0(\mathbb R^n)$},\]
and therefore $T$ coincides with $\mu$ treated as a distribution. Since
\[\|\mu_j\|_\alpha=\|\mu_j\|_{S_\alpha^*},\]
letting here $j\to\infty$ establishes (\ref{D-D}).\end{proof}

\begin{remark}If we restrict ourselves to measures of compact support, then relation (\ref{D-D}) does hold for any $\alpha\in(0,n)$ and the following stronger assertion is actually valid.

\begin{theorem}\label{S-D-comp}If\/ $\mu\in\mathfrak M(\mathbb R^n)$ has compact support then for any\/ $\alpha\in(0,n)$ we have\/ $\mu\in S_\alpha^*$ if and only if\/ $\mu\in\dot{\mathcal E}_\alpha(\mathbb R^n)$, and in the affirmative case\/ $\|\mu\|_\alpha^\cdot=\|\mu\|_{S_\alpha^*}$.\end{theorem}

\begin{proof} The function $\kappa_{\alpha/2}(x)=|x|^{\alpha/2-n}$ on $\mathbb R^n$ is locally $m$-integrable, and it therefore equals the absolutely continuous measure $\kappa_{\alpha/2}(x)\,dm(x)$. It follows from \cite[Theorem~0.10]{L} that this measure is a tempered distribution, i.e.\ an element of $S^*$. Again by the quoted theorem, so is the given measure $\mu$ of compact support, viewed as a distribution. Applying \cite[Corollary to Theorem~0.12]{L} to $\kappa_{\alpha/2}$ and $\mu$ from $S^*$, we get
in view of \cite[Eq.~1.1.1]{L} \[\mathcal F[\kappa_{\alpha/2}\ast\mu]=\mathcal F[\kappa_{\alpha/2}]\mathcal F[\mu]=\sqrt{C_{n,\alpha}}\,\frac{\mathcal F[\mu](\xi)}{|\xi|^{\alpha/2}},\text{ \ where\ }\xi\in\mathbb R^n.\]
By Plancherel's theorem the function on the right is square integrable if and only if $\kappa_{\alpha/2}\mu\in L^2(m)$, and then the Lebesgue integrals of the squares are equal.
This completes the proof.\end{proof}
\end{remark}

\section{A relation between the weak $\alpha$-Riesz and the $\alpha$-Green energies}\label{deny}

Returning to Section~\ref{sec:RG} with the domain $D\subset\mathbb R^n$ and its complement $D^c$ with $c_\alpha(D^c)>0$, we denote by $\mu'=\mu^{D^c}$ the $\alpha$-Riesz balayage of an extendible (signed) measure $\mu\in\mathfrak M(D)$ onto $D^c$, and by $g=g^\alpha_D$ the $\alpha$-Green kernel on $D$.

\begin{theorem}\label{thm} Let\/ $\mu\in\mathcal E_g(D)$ be an extendible\/ {\rm(}signed\/{\rm)} Radon measure on\/ $D$. Then\/ $\mu-\mu'$ is a\/ {\rm(}signed\/{\rm)} Radon measure on\/ $\mathbb R^n$ of finite weak\/ $\alpha$-Riesz energy\/ $\dot{E}_\alpha(\mu-\mu')$, and
\begin{equation}\label{2.8} E_g(\mu)=\dot{E}_\alpha(\mu-\mu').
\end{equation}
If\/ $\mu\in\mathcal E_\alpha(\mathbb R^n)$, i.e.\ if the extension of\/ $\mu$ has finite\/ $\alpha$-Riesz energy in the standard sense, then
\[E_g(\mu)=E_\alpha(\mu-\mu')=E_\alpha(\mu)-E_\alpha(\mu').\]
\end{theorem}

\begin{proof}Assume first that the given $\mu$ is positive.  According to Lemma~\ref{l-hatg}, $g\mu$ is given by (\ref{hatg}). Besides, since $E_g(\mu)<\infty$, $g\mu$ is finite $\mu$-a.e.\ on $D$. Choose an increasing sequence $\{K_j\}_{j\in\mathbb N}$ of compact subsets of $D$ with the union $D$ and write $\mu_j:=\mu|_{K_j}$. Since $S^{\mu_j}_D$ is compact and $E_g(\mu_j)<\infty$, it follows from Lemmas~\ref{l-hen'} and~\ref{l-hen'-comp} and equality (\ref{dot0}) that
\begin{equation}\label{1}\|\mu_j\|_g^2=\|\mu_j-\mu_j'\|_\alpha^2=\int[\kappa_{\alpha/2}(\mu_j-\mu_j')]^2\,dm.
\end{equation}

It is clear that $g\mu_j\uparrow g\mu$ pointwise on $D$, and also that $\|\mu_j\|_g\uparrow\|\mu\|_g<\infty$. The former relation implies that $\langle\mu_j,\mu_p\rangle_g\geqslant\|\mu_p\|^2_g$ for all $j\geqslant p$, and hence
\[\|\mu_j-\mu_p\|^2_g\leqslant\|\mu_j\|^2_g-\|\mu_p\|^2_g,\]
which together with the latter relation proves that $\{\mu_j\}_{j\in\mathbb N}$ is strong Cauchy in $\mathcal E^+_g(D)$. Also noting that $\mu_j\to\mu$ vaguely in $\mathfrak M(D)$, we conclude by the perfectness of $g$ \cite[Theorem~4.11]{FZ} that $\mu_j\to\mu$ in $\mathcal E_g(D)$ strongly (compare with \cite[Proposition~4]{Ca}).

For the proof of (\ref{2.8}) we shall show that $\kappa_{\alpha/2}(\mu-\mu')\in L^2(m)$ and that $\kappa_{\alpha/2}(\mu_j-\mu_j')\to\kappa_{\alpha/2}(\mu-\mu')$ in $L^2(m)$. For $j,k\in\mathbb N$ we obtain similarly as
in (\ref{1})
\[\|\mu_j-\mu_k\|_g^2
  =\|(\mu_j-\mu_j')-(\mu_k-\mu_k')\|_\alpha^2
  =\int[\kappa_{\alpha/2}((\mu_j-\mu_j')-(\mu_k-\mu_k'))]^2\,dm.\]
Since $\|\mu_j-\mu_k\|_g\to0$ as $j,k\to\infty$ this implies that the functions $\kappa_{\alpha/2}(\mu_j-\mu_j')$ form a Cauchy sequence in $L^2(m)$ and hence converge in $L^2(m)$-norm to some $\psi\in L^2(m)$. After passing to a subsequence we may assume that $\kappa_{\alpha/2}(\mu_j-\mu_j')\to\psi$ $m$-a.e.\ on~$\mathbb R^n$.

It follows from the definition of $\mu_j$ that $\kappa_\alpha\mu_j\uparrow\kappa_\alpha\mu$ pointwise on $\mathbb R^n$ and also that the increasing sequence $\{\mu_j\}_{j\in\mathbb N}$ converges to $\mu$ vaguely in $\mathfrak M(\mathbb R^n)$. We therefore see from the proof of \cite[Theorem~3.6]{FZ} that $\{\mu'_j\}_{j\in\mathbb N}$ is likewise increasing and converges vaguely to $\mu'$, which implies that $\kappa_{\alpha/2}\mu_j'\uparrow\kappa_{\alpha/2}\mu'$ pointwise on $\mathbb R^n$. Since
$\kappa_{\alpha/2}\mu_j\uparrow\kappa_{\alpha/2}\mu$ pointwise on $\mathbb R^n$, we thus have $\kappa_{\alpha/2}(\mu_j-\mu_j')\to\kappa_{\alpha/2}(\mu-\mu')$ pointwise $c_{\alpha/2}$-n.e.\ on $\mathbb R^n$, noting that $\kappa_{\alpha/2}\mu$ is finite $c_{\alpha/2}$-n.e.\ on $\mathbb R^n$ (see footnote~\ref{foot1}). Therefore $\psi=\kappa_{\alpha/2}(\mu-\mu')$ $m$-a.e.\ on $\mathbb R^n$, and hence $\mu-\mu'\in{\dot{\mathcal E}}_\alpha(\mathbb R^n)$ and $\kappa_{\alpha/2}(\mu_j-\mu_j')\to\kappa_{\alpha/2}(\mu-\mu')$ in $L^2(m)$.
Letting now $j\to\infty$ in (\ref{1}) and recalling that $E_g(\mu_j)\to E_g(\mu)$ we arrive at~(\ref{2.8}), so far for an extendible $\mu\in\mathcal E_g^+(D)$.

In order to establish (\ref{2.8}) for an extendible (signed) $\mu\in\mathcal E_g(D)$, we next show that
\begin{equation}\label{cor1}\langle\mu,\nu\rangle_g=\langle\mu-\mu',\nu-\nu'\rangle_\alpha^\cdot\text{ \ for any extendible \ }\mu,\nu\in\mathcal E^+_g(D).\end{equation}
Since $\mu+\nu\in\mathcal E_g^+(D)$ is likewise extendible, we obtain from~(\ref{2.8})
\begin{align*}\|\mu+\nu\|_g^2&=\|(\mu+\nu)-(\mu'+\nu')\|_\alpha^{\cdot\,2}\\
&{}=\|\mu-\mu'\|_\alpha^{\cdot\,2}+\|\nu-\nu'\|_\alpha^{\cdot\,2}+2\langle\mu-\mu',\nu-\nu'\rangle_\alpha^\cdot.
\end{align*}
On the other hand,
\begin{align*}\|\mu+\nu\|_g^2&=\|\mu\|_g^2+\|\nu\|_g^2+2\langle\mu,\nu\rangle_g\\
&{}=\|\mu-\mu'\|_\alpha^{\cdot\,2}+\|\nu-\nu'\|_\alpha^{\cdot\,2}+2\langle\mu,\nu\rangle_g.
\end{align*}
Comparing the resulting equations in the above two displays yields (\ref{cor1}). (This can be extended to signed extendible $\mu,\nu\in\mathcal E_g(D)$ by using the bilinearity of an inner product and the linearity of balayage after having inserted $\mu=\mu^+-\mu^-$ and $\nu=\nu^+-\nu^-$ into~(\ref{cor1}).)

If now $\mu\in\mathcal E_g(D)$ is an extendible (signed) measure on $D$, then we obtain from the above
\begin{align*}\|\mu\|_g^2&=\|\mu^+-\mu^-\|_g^2=\|\mu^+\|_g^2+\|\mu^-\|_g^2-2\langle\mu^+,\mu^-\rangle_g\\
&{}=\|\mu^+-(\mu^+)'\|_\alpha^{\cdot\,2}+\|\mu^--(\mu^-)'\|_\alpha^{\cdot\,2}-2\bigl\langle\mu^+-(\mu^+)',\mu^--(\mu^-)'\bigr\rangle_\alpha^\cdot\\
&{}=\|\mu^+-(\mu^+)'-\bigl(\mu^--(\mu^-)'\bigr)\|_\alpha^{\cdot\,2}=\|(\mu^+-\mu^-)-(\mu^+-\mu^-)'\|_\alpha^{\cdot\,2}=\|\mu-\mu'\|_\alpha^{\cdot\,2},
\end{align*}
which shows that for an extendible $\mu\in\mathcal E_g(D)$, $\dot{E}_\alpha(\mu-\mu')$ is finite and (\ref{2.8}) holds.

Under the extra requirement $E_\alpha(\mu)<\infty$ we get the last assertion of the theorem from Lemma~\ref{l-hen'}. However, under the stated assumption $E_g(\mu)<\infty$ this extra requirement is not fulfilled in general, as shown by the example in \cite[Appendix]{DFHSZ2}. Combined with the present theorem, the quoted example also shows that a bounded (signed) measure (here $\mu-\mu'$) may be of class $\dot{\mathcal E}_\alpha(\mathbb R^n)$ but not of class $\mathcal E_\alpha(\mathbb R^n)$, as mentioned in the Introduction.
\end{proof}

\begin{remark}If $\alpha=2$, then  $S_{\mathbb R^n}^{\nu'}=\partial_{\mathbb R^n}D$ for any extendible $\nu\in\mathfrak M^+(D)$ \cite[Proof of Theorem~6.4]{DFHSZ2}. Combined with Theorems~\ref{S-D} and~\ref{thm}, the example in \cite[Appendix]{DFHSZ2} therefore implies that there exists a linear combination of positive measures with infinite standard $\alpha$-Riesz energy and even with compact support, whose energy in $S_\alpha^*$ is finite. This gives an answer in the negative to the question raised by Deny in \cite[p.~125, Remarque]{De1}.\end{remark}

Let $F$ be a relatively closed subset of $D$ with $w_g(F)<\infty$, and let $\mathbb M_g(F,1;D)$ consist of all (minimizing) sequences $\{\nu_j\}_{j\in\mathbb N}\subset\mathcal E^+_g(F,1;D)$ possessing the property
\begin{equation}\label{g-min-seq}\lim_{j\to\infty}\,\|\nu_j\|^2_g=w_g(F)=1/c_g(F).\end{equation}
The class $\mathbb M_g(F,1;D)$ is nonempty by \cite[Lemma~2.3.1]{F1} with $\kappa=g$.

\begin{theorem}\label{th-aux}Assume moreover that\/ $c_g(F)<\infty$ and that\/ $D^c$ is not\/ $\alpha$-thin at infinity.
For any\/ $\{\nu_j\}_{j\in\mathbb N}\in\mathbb M_g(F,1;D)$ then the following two assertions on convergence hold:
\begin{itemize}\item[{\rm(a)}] $\|\nu_j-\lambda_{F,g}\|_g\to0$ as\/ $j\to\infty$, where\/ $\lambda_{F,g}$ is the\/ {\rm(}unique{\rm)} $g$-capacitary measure on\/ $F$ {\rm(}see Remark\/~{\rm\ref{remark}}{\rm)}, and hence\/ $\nu_j\to\lambda_{F,g}$ vaguely in\/ $\mathfrak M^+(D)$.
\item[{\rm(b)}] $\nu_j-\nu_j'\to\lambda_{F,g}-\lambda'_{F,g}$ in the weak energy norm\/ $\|\cdot\|^\cdot_\alpha$.
\end{itemize}
Moreover,  there exists a\/ {\rm(}particular\/{\rm)} minimizing sequence\/ $\{\tilde{\nu}_j\}_{j\in\mathbb N}\in\mathbb M_g(F,1;D)$ such that, in addition to\/ {\rm(a)} and\/ {\rm(b)}, $E_\alpha(\tilde{\nu}_j-\tilde{\nu}_j')<\infty$ and furthermore
\begin{itemize}
\item[{\rm(c)}] $\{\tilde{\nu}_j\}_{j\in\mathbb N}$ and\/ $\{\tilde{\nu}'_j\}_{j\in\mathbb N}$ converge vaguely in\/ $\mathfrak M^+(\mathbb R^n)$ to\/ $\lambda_{F,g}$ and\/ $\lambda'_{F,g}$, respectively, and hence\/ $\tilde{\nu}_j-\tilde{\nu}_j'\to\lambda_{F,g}-\lambda'_{F,g}$ vaguely in\/ $\mathfrak M(\mathbb R^n)$.
\end{itemize}
\end{theorem}

\begin{proof} It follows by standard arguments based on the convexity of $\mathcal E^+_g(F,1;D)$ and the pre-Hilbert structure on $\mathcal E_g(D)$ that for any two  $\{\nu_j\}_{j\in\mathbb N}$ and $\{\mu_j\}_{j\in\mathbb N}$ in $\mathbb M_g(F,1;D)$,
\begin{equation}\label{fund}\lim_{j\to\infty}\,\|\nu_j-\mu_j\|_g=0.\end{equation}
In particular, this implies that every such $\{\nu_j\}_{j\in\mathbb N}$ is strong Cauchy in $\mathcal E^+_g(F;D)$.

Define $\lambda_j:=\lambda|_{K_j}$ where $K_j$, $j\in\mathbb N$, are as in the proof of Theorem~\ref{thm} and $\lambda:=\lambda_{F,g}$. As has been shown in the quoted proof,
$\lambda_j\to\lambda$ strongly in $\mathcal E^+_g(D)$, $\lambda_j-\lambda'_j\to\lambda-\lambda'$ in the weak energy norm $\|\cdot\|^\cdot_\alpha$, and furthermore $\lambda_j$ and $\lambda_j'$ converge vaguely in $\mathfrak M^+(\mathbb R^n)$ to $\lambda$ and $\lambda'$, respectively. Moreover, $E_\alpha(\lambda_j)$ is finite, and hence so is $E_\alpha(\lambda_j-\lambda'_j)$. Since $\lambda_j(K_j)\to\lambda(F)=1$ and since, according to Theorem~\ref{bal-mass-th}, $\mu'(\mathbb R^n)=\mu(\mathbb R^n)$ for every bounded $\mu\in\mathfrak M^+(D)$, we thus infer that $\{\tilde{\nu}_j\}_{j\in\mathbb N}$ with $\tilde{\nu}_j:=\lambda_j/\lambda_j(K_j)$ is a particular element of $\mathbb M_g(F,1;D)$ which satisfies all the assertions stated in the theorem.

Now fix any $\{\nu_j\}_{j\in\mathbb N}\in\mathbb M_g(F,1;D)$.
It follows from (\ref{fund}) (with $\tilde{\nu}_j$ in place of $\mu_j$) together with (a) applied to $\{\tilde{\nu}_j\}_{j\in\mathbb N}$ that $\{\nu_j\}_{j\in\mathbb N}$ converges to $\lambda$ strongly in $\mathcal E^+_g(D)$. Since by Theorem~\ref{fu-complete} with $\kappa=g$ the strong topology on $\mathcal E^+_g(D)$ is finer than the induced vague topology, $\nu_j\to\lambda$ also vaguely in $\mathfrak M^+(D)$, which establishes (a). Finally, by Theorem~\ref{thm} each $\nu_j-\nu_j'$, $j\in\mathbb N$, belongs to $\dot{\mathcal E}_\alpha(\mathbb R^n)$ and, furthermore,
\[\|(\nu_j-\nu'_j)-(\lambda-\lambda')\|_\alpha^\cdot=\|(\nu_j-\lambda)-(\nu_j-\lambda)'\|_\alpha^\cdot=\|\nu_j-\lambda\|_g,\]
which in view of (a) implies~(b).
\end{proof}

\section{Applications to generalized condensers}

\subsection{Relevant notions and problems in condenser theory} We first define relevant notions for condensers in $\mathbb R^n$ and formulate the corresponding problems.

\begin{definition}[see \cite{DFHSZ2}]\label{def-g-c} An ordered pair ${\mathbf A}=(A_1,A_2)$ is termed a ({\it generalized\/}) {\it condenser\/} in $\mathbb R^n$ if $A_1$ is a relatively closed subset of a domain $D\subset\mathbb R^n$ and  $A_2=D^c$. The sets $A_1$ and $A_2$ are said to be the {\it positive and negative plates\/}, respectively. A (generalized) condenser $\mathbf A$ is termed {\it standard\/} if $A_1$ is closed in $\mathbb R^n$.\end{definition}

Given a (generalized) condenser ${\mathbf A}=(A_1,A_2)$, let $\mathfrak M({\mathbf A};\mathbb R^n)$ stand for the class of all (signed) Radon measures $\mu\in\mathfrak M(\mathbb R^n)$ such that $\mu^+\in\mathfrak M^+(A_1;\mathbb R^n)$ and $\mu^-\in\mathfrak M^+(A_2;\mathbb R^n)$, and let $\mathfrak M({\mathbf A},{\mathbf 1};\mathbb R^n)$ consist of all $\mu\in\mathfrak M({\mathbf A};\mathbb R^n)$ with $\mu^+(A_1)=\mu^-(A_2)=1$. To avoid trivialities, assume that
\begin{equation}\label{nonzero}c_\alpha(A_i)>0\text{ \ for all \ }i=1,2.\end{equation}
Then $\mathcal E_\alpha({\mathbf A},{\mathbf 1};\mathbb R^n):=\mathcal E_\alpha(\mathbb R^n)\cap\mathfrak M({\mathbf A},{\mathbf 1};\mathbb R^n)$ is not empty \cite[Lemma~2.3.1]{F1}, and hence so is $\dot{\mathcal E}_\alpha({\mathbf A},{\mathbf 1};\mathbb R^n):=\dot{\mathcal E}_\alpha(\mathbb R^n)\cap\mathfrak M({\mathbf A},{\mathbf 1};\mathbb R^n)$. Write
\begin{align*}\dot{\mathcal E}^\circ_\alpha({\mathbf A},{\mathbf 1};\mathbb R^n)&:=\dot{\mathcal E}_\alpha({\mathbf A},{\mathbf 1};\mathbb R^n)\cap C\ell_{\dot{\mathcal E}_\alpha(\mathbb R^n)}\,\mathcal E_\alpha({\mathbf A},{\mathbf 1};\mathbb R^n),\\
\notag w_\alpha(\mathbf A)&:=\inf_{\mu\in\mathcal E_\alpha({\mathbf A},{\mathbf 1};\mathbb R^n)}\,E_\alpha(\mu),\\
\notag {\dot w}_\alpha(\mathbf A)&:=\inf_{\nu\in\dot{\mathcal E}^\circ_\alpha({\mathbf A},{\mathbf 1};\mathbb R^n)}\,{\dot E}_\alpha(\nu).\end{align*}

\begin{definition}\label{def-R-c}$c_\alpha(\mathbf A):=1/w_\alpha(\mathbf A)$, resp.\ ${\dot c}_\alpha(\mathbf A):=1/{\dot w}_\alpha(\mathbf A)$, is said to be the {\it standard\/}, resp.\ the {\it weak\/}, {\it $\alpha$-Riesz capacity\/} of a (generalized) condenser~$\mathbf A$.\end{definition}

Since $w_\alpha(\mathbf A)$ and ${\dot w}_\alpha(\mathbf A)$ are both finite, the following minimum standard, resp.\ weak, $\alpha$-Riesz energy problem makes sense.

\begin{problem}\label{pr-st}Does there exist $\lambda_{{\mathbf A},\alpha}\in\mathcal E_\alpha({\mathbf A},{\mathbf 1};\mathbb R^n)$ with $E_\alpha(\lambda_{{\mathbf A},\alpha})= w_\alpha(\mathbf A)$?\end{problem}

\begin{problem}\label{pr-weak}Does there exist $\dot{\lambda}_{{\mathbf A},\alpha}\in\dot{\mathcal E}^\circ_\alpha({\mathbf A},{\mathbf 1};\mathbb R^n)$ with ${\dot E}_\alpha(\dot{\lambda}_{{\mathbf A},\alpha})={\dot w}_\alpha(\mathbf A)$?\end{problem}

\begin{lemma} A solution to either Problem\/~{\rm\ref{pr-st}} or Problem\/~{\rm\ref{pr-weak}} is unique\/ {\rm(}provided it exists\/{\rm)}.\end{lemma} 

\begin{proof}This can be established by standard methods based on the convexity of $\mathcal E_\alpha({\mathbf A},{\mathbf 1};\mathbb R^n)$, resp.\ $\dot{\mathcal E}^\circ_\alpha({\mathbf A},{\mathbf 1};\mathbb R^n)$, and the pre-Hilbert structure on $\mathcal E_\alpha(\mathbb R^n)$, resp.\ $\dot{\mathcal E}_\alpha(\mathbb R^n)$.\end{proof}

\begin{remark}\label{r-2}If $\mathbf A$ is a standard condenser with compact $A_i$, $i=1,2$, then the solvability  of Problem~\ref{pr-st} can easily be established by exploiting the vague topology only, since then $\mathfrak M({\mathbf A},{\mathbf 1};\mathbb R^n)$ is vaguely compact, while $E_\alpha(\cdot)$ is va\-guely l.s.c.\ on $\mathcal E_\alpha({\mathbf A};\mathbb R^n):=\mathcal E_\alpha(\mathbb R^n)\cap\mathfrak M({\mathbf A};\mathbb R^n)$ (see e.g.\ \cite[Theorem~2.30]{O}). However, these arguments break down if any of the $A_i$ is noncompact in $\mathbb R^n$, for then $\mathfrak M({\mathbf A},{\mathbf 1};\mathbb R^n)$ is no longer vaguely compact.\end{remark}

\begin{remark}\label{r-3}Let $\mathbf A$ be a standard condenser possessing the property (\ref{dist}).
Under these assumptions, in \cite{ZPot1,ZPot2} an approach has been worked out based on both the vague and the strong topologies on $\mathcal E_\alpha({\mathbf A};\mathbb R^n)$ which made it possible to show that the requirement
\begin{equation}\label{r-suff}c_\alpha(A_i)<\infty\text{ \ for all \ }i=1,2\end{equation}
is sufficient for Problem~\ref{pr-st} to be solvable \cite[Theorem~8.1]{ZPot2}. (Compare with Theorem~\ref{th-st} and Section~\ref{rem-st} below.) However, if assumption (\ref{dist}) is omitted, then the approach developed in \cite{ZPot1,ZPot2} breaks down, and (\ref{r-suff}) no longer guarantees the existence of a solution to Problem~\ref{pr-st}. This has been illustrated by \cite[Theorem~4.6]{DFHSZ1} pertaining to the Newtonian kernel. A solution to the minimum standard $\alpha$-Riesz energy problem nevertheless does exist if we require additionally that the positive and the negative parts of the admissible measures are majorized by properly chosen constraints \cite{DFHSZ1,DFHSZ2}.\end{remark}

\begin{definition}\label{def-m-c}$\mu=\mu_{{\mathbf A},\alpha}\in\mathfrak M({\mathbf A};\mathbb R^n)$ is said to be a {\it measure of a generalized condenser\/} $\mathbf A$ relative to the $\alpha$-Riesz kernel $\kappa_\alpha$ (or briefly a {\it condenser measure\/}) if $\kappa_\alpha\mu$ takes the value $1$ and $0$ $c_\alpha$-n.e.\ on $A_1$ and $A_2$, respectively, and $0\leqslant \kappa_\alpha\mu\leqslant1$ $c_\alpha$-n.e.\ on $\mathbb R^n$.\end{definition}

In the case where $A_1$ and $A_2$ are compact disjoint sets the existence of a condenser measure was established by Kishi~\cite{Ki}, actually even in the general setting of a function kernel on a locally compact Hausdorff space. See also \cite{D3}, \cite{L}, \cite{Bl}, \cite{Berg} where the existence of condenser potentials was analyzed in the framework of Dirichlet spaces.

\subsection{Auxiliary results} Lemmas~\ref{l:dot} and \ref{eq-g-r} below establish relations between $c_\alpha(\mathbf A)$, ${\dot c}_\alpha(\mathbf A)$ and $c_g(A_1)$, where $g=g^\alpha_D$ is the $\alpha$-Green kernel on the domain $D=A_2^c$.

\begin{lemma}\label{l:dot} $c_\alpha(\mathbf A)={\dot c}_\alpha(\mathbf A)$.
\end{lemma}

\begin{proof}Since $\mathcal E_\alpha({\mathbf A},{\mathbf 1};\mathbb R^n)\subset\dot{\mathcal E}^\circ_\alpha({\mathbf A},{\mathbf 1};\mathbb R^n)$, we have $0\leqslant{\dot w}_\alpha(\mathbf A)\leqslant w_\alpha(\mathbf A)<\infty$ by (\ref{dot0}). Thus it is enough to show that $w_\alpha(\mathbf A)\leqslant{\dot w}_\alpha(\mathbf A)$.
Choose $\{\nu_j\}_{j\in\mathbb N}\subset\dot{\mathcal E}^\circ_\alpha({\mathbf A},{\mathbf 1};\mathbb R^n)$ with
\begin{equation}\label{limdot}\lim_{j\to\infty}\,{\dot E}_\alpha(\nu_j)={\dot w}_\alpha(\mathbf A),\end{equation}
and for every $\nu_j$ choose $\mu_j\in\mathcal E_\alpha({\mathbf A},{\mathbf 1};\mathbb R^n)$ so that $\|\nu_j-\mu_j\|_\alpha^\cdot<j^{-1}$. Applying (\ref{dot0}) to $\mu_j$, we obtain from the triangle inequality in the pre-Hilbert space $\dot{\mathcal E}_\alpha(\mathbb R^n)$
\[w_\alpha(\mathbf A)\leqslant E_\alpha(\mu_j)={\dot E}_\alpha(\mu_j)={\dot E}_\alpha(\mu_j-\nu_j+\nu_j)\leqslant\bigl(j^{-1}+\|\nu_j\|^\cdot_\alpha\bigr)^2.\]
In view of (\ref{limdot}) we complete the proof by letting here $j\to\infty$.
\end{proof}

\begin{lemma}\label{eq-g-r}If\/ $A_2$ is not\/ $\alpha$-thin at infinity, then
\begin{equation*}\label{eq3}c_\alpha(\mathbf A)=c_g(A_1).\end{equation*}
\end{lemma}

\begin{proof} We first observe that, by (\ref{nonzero}) and \cite[Lemma~2.6]{DFHSZ}, $c_g(A_1)>0$.
Consider an exhaustion of $A_1$ by compact sets $K_j\uparrow A_1$, $j\in\mathbb N$. Since $c_g(K_j)\uparrow c_g(A_1)$
by (\ref{compact}), there is no loss of generality in assuming $c_g(K_j)>0$, $j\in\mathbb N$. Furthermore, $c_g(K_j)<\infty$ because of the strict positive definiteness of the kernel $g$. In view of the perfectness of $g$, we thus see from Remark~\ref{remark} that there exists a (unique) $g$-capacitary measure $\lambda_j=\lambda_{K_j,g}$ on $K_j$, i.e.\ $\lambda_j\in\mathcal E^+_g(K_j,1;D)$ such that
\[w_g(K_j)=c_g(K_j)^{-1}=\|\lambda_j\|_g^2.\]
According to Lemmas~\ref{l-hen'} and \ref{l-hen'-comp}, $E_\alpha(\lambda_j)$ is finite along $E_g(\lambda_j)$ and
\[\|\lambda_j\|_g^2=\|\lambda_j-\lambda_j'\|_\alpha^2,\]
where $\lambda_j'$ is the $\alpha$-Riesz balayage of $\lambda_j$ onto $A_2$. As $A_2$ is not $\alpha$-thin at infinity, $\lambda_j'(A_2)=\lambda_j(K_j)$ by Theorem~\ref{bal-mass-th}, and hence $\lambda_j-\lambda_j'\in\mathcal E_\alpha({\mathbf A},{\mathbf 1};\mathbb R^n)$. Combined with the preceding two displays, this yields $w_g(K_j)\geqslant w_\alpha(\mathbf A)$,
which leads to $w_g(A_1)\geqslant w_\alpha(\mathbf A)$ when $j\to\infty$.

For the converse inequality, choose $\{\mu_j\}_{j\in\mathbb N}\subset\mathcal E_\alpha({\mathbf A},{\mathbf 1};\mathbb R^n)$ with
\begin{equation}\label{min-seq}\lim_{j\to\infty}\,\|\mu_j\|^2_\alpha=w_\alpha(\mathbf A).\end{equation}
Since the $\alpha$-Riesz balayage $\mu'$ of $\mu\in\mathcal E_\alpha^+(\mathbb R^n)$ onto $A_2$ is in fact the orthogonal projection in the pre-Hilbert space $\mathcal E_\alpha(\mathbb R^n)$ of $\mu$ onto the convex cone $\mathcal E^+_\alpha(A_2;\mathbb R^n)$ (see (\ref{proj})), we get
\[\|\mu_j\|^2_\alpha\geqslant\|\mu_j^+-(\mu_j^+)'\|^2_\alpha=\|\mu_j^+\|^2_g\geqslant w_g(A_1),\]
the equality being valid by Lemma~\ref{l-hen'}. Letting here $j\to\infty$, in view of (\ref{min-seq}) we arrive at $w_g(A_1)\leqslant w_\alpha(\mathbf A)$, thus completing the proof.
\end{proof}

\begin{lemma}\label{l:w:min}Assume that\/ $c_g(A_1)<\infty$ and\/ $A_2$ is not\/ $\alpha$-thin at infinity. Then for any\/ $\{\nu_j\}_{j\in\mathbb N}\in\mathbb M_g(A_1,1;D)$ {\rm(}see Theorem\/~{\rm\ref{th-aux}} with\/ $A_1$ in place of\/ $F${\rm)} we have
\begin{equation}\label{in:dot}\nu_j-\nu_j'\in\dot{\mathcal E}^\circ_\alpha({\mathbf A},{\mathbf 1};\mathbb R^n),\end{equation}
\begin{equation}\label{lim:dot}
\lim_{j\to\infty}\,\|\nu_j-\nu_j'\|_\alpha^{\cdot\,2}=\lim_{j\to\infty}\,\|\nu_j\|_g^2=w_g(A_1)=w_\alpha(\mathbf A)=\dot{w}_\alpha(\mathbf A).
\end{equation}
\end{lemma}

\begin{proof} Theorem~\ref{thm} with $\nu_j$ in place of $\mu$ shows that $\nu_j-\nu_j'\in\dot{\mathcal E}_\alpha(\mathbb R^n)$ and
$\|\nu_j\|_g=\|\nu_j-\nu_j'\|_\alpha^\cdot$, which in view of (\ref{g-min-seq}) establishes (\ref{lim:dot}) if combined with Lemmas~\ref{l:dot} and~\ref{eq-g-r}. Applying now to $\nu_j$ the same arguments as in the second paragraph of the proof of Theorem~\ref{th-aux} with $\nu_j$ in place of $\lambda$, we see that $\nu_j-\nu_j'$ can be approximated in the weak energy norm $\|\cdot\|^\cdot_\alpha$ by elements of the class $\mathcal E_\alpha({\mathbf A},{\mathbf 1};\mathbb R^n)$, and (\ref{in:dot}) follows.\end{proof}

Under the hypotheses of Lemma~\ref{l:w:min} we denote by $\dot{\mathbb M}_\alpha({\mathbf A},{\mathbf 1};\mathbb R^n)$  the collection of all sequences $\{\nu_j-\nu_j'\}_{j\in\mathbb N}$ where $\{\nu_j\}_{j\in\mathbb N}$ ranges over $\mathbb M_g(A_1,1;D)$.

\subsection{Main results} Theorem~\ref{th-ex} below provides necessary and sufficient conditions for the solvability of Problem~\ref{pr-weak} and establishes an intimate relationship between $\dot{\lambda}_{{\mathbf A},\alpha}$, a condenser measure $\mu_{{\mathbf A},\alpha}$, and the $g$-capacitary measure on~$A_1$.

\begin{theorem}\label{th-ex}Given a generalized condenser\/ $\mathbf A$, assume that\/ $A_2$ is not\/ $\alpha$-thin at infinity. Then the following three assertions are equivalent:
\begin{itemize}
\item[{\rm (i)}] There exists a\/ {\rm(}unique\/{\rm)} solution\/ $\dot{\lambda}_{{\mathbf A},\alpha}$ to Problem\/~{\rm\ref{pr-weak}}.
\item[{\rm (ii)}] There exists a\/ {\rm(}unique\/{\rm)} bounded\/ $c_\alpha$-absolutely continuous condenser measure\/ $\mu_{{\mathbf A},\alpha}$.
\item[{\rm (iii)}] $c_g(A_1)<\infty$.
\end{itemize}
If any of\/ {\rm(i)}--{\rm(iii)} is valid, then\/ $\dot{\lambda}_{{\mathbf A},\alpha}$ can be written in the form
\begin{equation}\label{repr2}\dot{\lambda}_{{\mathbf A},\alpha}=\lambda_{A_1,g}-\lambda_{A_1,g}',\end{equation}
where\/ $\lambda_{A_1,g}$ is the\/ {\rm(}unique\/{\rm)} $g$-capacitary measure on\/ $A_1$ and\/ $\lambda_{A_1,g}'$ its\/ $\alpha$-Riesz balayage onto\/ $A_2$, while
\begin{equation}\label{repr222}\mu_{{\mathbf A},\alpha}=c_g(A_1)\cdot\dot{\lambda}_{{\mathbf A},\alpha}.\end{equation}
Furthermore,
\begin{equation}\label{eq2}{\dot E}_\alpha(\mu_{{\mathbf A},\alpha})=c_g(A_1)=c_\alpha(\mathbf A)=\dot{c}_\alpha(\mathbf A)<\infty,\end{equation}
and for elements of\/ $\dot{\mathbb M}_\alpha({\mathbf A},{\mathbf 1};\mathbb R^n)$ the following two assertions on convergence hold:
\begin{itemize}
\item[{\rm(a$'$)}] For any\/ $\{\nu_j-\nu_j'\}_{j\in\mathbb N}\in\dot{\mathbb M}_\alpha({\mathbf A},{\mathbf 1};\mathbb R^n)$, $\nu_j-\nu_j'\to\dot{\lambda}_{{\mathbf A},\alpha}$ in the weak energy norm\/~$\|\cdot\|^\cdot_\alpha$.
\item[{\rm(b$'$)}] There exists\/ $\{\tilde{\nu}_j-\tilde{\nu}_j'\}_{j\in\mathbb N}\in\dot{\mathbb M}_\alpha({\mathbf A},{\mathbf 1};\mathbb R^n)$ such that\/ $\tilde{\nu}_j\to\lambda_{A_1,g}$ and\/ $\tilde{\nu}_j'\to\lambda_{A_1,g}'$ vaguely in\/ $\mathfrak M^+(\mathbb R^n)$, and hence\/ $\tilde{\nu}_j-\tilde{\nu}_j'\to\dot{\lambda}_{{\mathbf A},\alpha}$ vaguely in\/ $\mathfrak M(\mathbb R^n)$. In addition,
        \begin{equation}\label{partic}\{\tilde{\nu}_j-\tilde{\nu}_j'\}_{j\in\mathbb N}\subset\mathcal E_\alpha({\mathbf A},{\mathbf 1};\mathbb R^n).\end{equation}
\end{itemize}
\end{theorem}

\begin{proof}Assume first that (iii) holds. By \cite[Theorem~4.12]{FZ}, then there exists the (unique) $g$-equilibrium measure on $A_1$, i.e.\ $\gamma=\gamma_{A_1,g}\in\mathcal E^+_g(A_1;D)$ with the properties
\begin{align}\label{totmass}&\|\gamma\|^2_g=\gamma(A_1)=c_g(A_1),\\
\label{pot1}&g\gamma=1\text{ \ $c_g$-n.e.\ on $A_1$},
\end{align}
which is determined uniquely among measures of the class $\mathcal E^+_g(A_1;D)$ by (\ref{pot1}). Since a nonzero positive measure of finite $g$-energy does not charge any set in $D$ with $c_g(\cdot)=0$, (\ref{pot1}) also holds $\gamma$-a.e. Applying Theorem~\ref{th-dom-pr}, we therefore get
\begin{equation}g\gamma\leqslant1\text{ \ on $D$}.\label{pot2}\end{equation}
We assert that $\mu:=\gamma-\gamma'$ serves as a bounded $c_\alpha$-absolutely continuous condenser measure $\mu_{{\mathbf A},\alpha}$.
Being bounded in view of (\ref{totmass}), the measure $\gamma$ is extendible, and Lemma~\ref{l-hatg} yields \begin{equation}\label{pot3}g\gamma=\kappa_\alpha(\gamma-\gamma')=\kappa_\alpha\mu\text{ \ $c_\alpha$-n.e.\ on $D$}.\end{equation}
Being $c_g$-absolutely continuous, $\gamma$ is $c_\alpha$-absolutely continuous (see footnote~\ref{RG}). Thus (\ref{pot1}) holds also $c_\alpha$-n.e., which together  with (\ref{pot3}) implies that $\kappa_\alpha\mu=1$ $c_\alpha$-n.e.\ on $A_1$. The relation $\kappa_\alpha\mu=0$ $c_\alpha$-n.e.\ on $A_2$ follows from (\ref{bal-eq}). We therefore see by combining (\ref{pot2}) with (\ref{pot3}) that $\gamma-\gamma'$ indeed serves as a condenser measure $\mu_{{\mathbf A},\alpha}$. Since the balayage $\gamma'$ is  $c_\alpha$-ab\-sol\-utely continuous and bounded along with $\gamma$ (see Section~\ref{sec:RG}), this $\mu_{{\mathbf A},\alpha}$ justifies (ii), and it is unique according to \cite[p.~178, Remark]{L}. Hence, (iii)$\Rightarrow$(ii).

We shall next show that (ii) implies (iii). It follows from the definition of a condenser measure $\mu_{{\mathbf A},\alpha}=\mu^+-\mu^-$ that $\mu^-$ is in fact the $\alpha$-Riesz balayage of $\mu^+$ onto $A_2$, i.e.
\begin{equation}\label{rep1}\mu_{{\mathbf A},\alpha}=\mu^+-(\mu^+)'.\end{equation}
Treating $\mu^+$ as an extendible measure from $\mathfrak M^+(D)$, we obtain from Lemma~\ref{l-hatg}
\[g\mu^+=\kappa_\alpha\bigl(\mu^+-(\mu^+)'\bigr)=\kappa_\alpha\mu_{{\mathbf A},\alpha}=1\text{ \ $c_\alpha$-n.e.\ on $A_1$}.\]
Integrating this equality with respect to the ($c_\alpha$-absolutely continuous) measure $\mu^+$ gives
\[E_g(\mu^+)=\int g\mu^+\,d\mu^+=\mu^+(A_1)<\infty,\]
the inequality being valid since $\mu_{{\mathbf A},\alpha}$ is by assumption bounded. 
Applying \cite[Lemma~3.2.2]{F1} with $\kappa=g$, we therefore see from the last display that $c_g(A_1)<\infty$, thus completing the proof that (ii) implies (iii).
Furthermore, it follows from the above that $\mu^+$ is in fact the $g$-equilibrium measure $\gamma=\gamma_{A_1,g}$ on $A_1$. Substituting $\mu^+=\gamma_{A_1,g}$ into (\ref{rep1}) implies
\begin{equation}\label{eqq}\mu_{{\mathbf A},\alpha}=\gamma_{A_1,g}-\gamma_{A_1,g}',\end{equation}
which in view of (\ref{2.8}) shows that
\[\dot{E}_\alpha(\mu_{{\mathbf A},\alpha})=E_g(\gamma)<\infty.\]
This establishes (\ref{eq2}) when combined with (\ref{totmass}) and Lemmas~\ref{l:dot} and~\ref{eq-g-r}.

If (i) holds, then $\dot{w}_\alpha(\mathbf A)=\dot{E}_\alpha(\dot{\lambda}_{{\mathbf A},\alpha})>0$ since $\dot{\lambda}_{{\mathbf A},\alpha}\ne0$. Thus $w_g(A_1)>0$ by Lemmas~\ref{l:dot} and~\ref{eq-g-r}, and so (i) implies (iii). For the converse implication, suppose now that (iii) holds. By (\ref{nonzero}) and \cite[Lemma~2.6]{DFHSZ} we have $c_g(A_1)>0$, and hence there exists the (unique) $g$-capacitary measure $\lambda=\lambda_{A_1,g}$ on $A_1$. The trivial sequence $\{\lambda\}_{j\in\mathbb N}$ belongs to $\mathbb M_g(A_1,1;D)$, and therefore we obtain from Lemma~\ref{l:w:min}
\[\lambda-\lambda'\in\dot{\mathcal E}^\circ_\alpha({\mathbf A},{\mathbf 1};\mathbb R^n).\]
Since obviously $\lambda=\gamma/c_g(A_1)$, we also have
\[\dot{E}_\alpha(\lambda-\lambda')=\frac{\dot{E}_\alpha(\gamma-\gamma')}{c_g(A_1)^2}=\frac{E_g(\gamma)}{c_g(A_1)^2}=w_g(A_1)=\dot{w}_\alpha(\mathbf A),\]
where the second equality holds by Theorem~\ref{thm}, the third by (\ref{totmass}), and the fourth by Lemmas~\ref{l:dot} and~\ref{eq-g-r}. Thus $\dot{\lambda}_{{\mathbf A},\alpha}:=\lambda-\lambda'$ is the (unique) solution to Problem~\ref{pr-weak}, and hence (iii) implies (i) with $\dot{\lambda}_{{\mathbf A},\alpha}$ given by (\ref{repr2}). Combining (\ref{repr2}) and (\ref{eqq}) establishes~(\ref{repr222}).

Finally, in view of (\ref{repr2}) we arrive at both (a$'$) and (b$'$) by combining Theorem~\ref{th-aux} and Lemma~\ref{l:w:min}, thus completing the proof of the theorem.
\end{proof}

Under the additional requirement that $A_1$ and $A_2$ are closed subsets of $\mathbb R^n$ with nonzero Euclidean distance, Theorem~\ref{th-ex} can be specified as follows.

\begin{theorem}\label{th-st}Assume that\/ $\mathbf A$ is a standard condenser in\/ $\mathbb R^n$ satisfying\/ {\rm(\ref{dist})}, and let\/ $A_2=D^c$ be not\/ $\alpha$-thin at infinity. Then any of the\/ {\rm(}equivalent\/{\rm)} assertions \mbox{{\rm(i)}--{\rm(iii)}} in Theorem\/~{\rm\ref{th-ex}} is equivalent to either of the following assertions:
\begin{itemize}
\item[{\rm (iv)}] There exists a\/ {\rm(}unique\/{\rm)} solution\/ $\lambda_{{\mathbf A},\alpha}$ to Problem\/~{\rm\ref{pr-st}}.
\item[{\rm (v)}] $c_\alpha(A_1)<\infty$.
\end{itemize}
If any of these\/ {\rm(i)}--{\rm(v)} holds, then the measures\/ $\dot{\lambda}_{{\mathbf A},\alpha}$ and\/ $\mu_{{\mathbf A},\alpha}$ from\/ {\rm(i)} and\/ {\rm(ii)}, respectively, now have finite standard\/ $\alpha$-Riesz energy and moreover
\begin{equation}\label{st-weak-sol}\lambda_{{\mathbf A},\alpha}=\dot{\lambda}_{{\mathbf A},\alpha}.\end{equation}
Furthermore, any\/ $\{\mu_k\}_{k\in\mathbb N}\subset\mathcal E_\alpha({\mathbf A},{\mathbf 1};\mathbb R^n)$ with the property\/ {\rm(\ref{min-seq})}
converges to\/ $\lambda_{{\mathbf A},\alpha}$ both strongly and vaguely in\/~$\mathcal E_\alpha(\mathbb R^n)$.
\end{theorem}

\begin{proof}Assume first that (v) holds. According to Lemma~\ref{eq-r-g}, (v) is equivalent to (iii), and hence also to either (i) or (ii) (see Theorem~\ref{th-ex} above). Furthermore, again by the quoted lemma, in view of (iii) the extension of the (bounded) $g$-equilibrium measure $\gamma=\gamma_{A_1,g}$  has finite $\alpha$-Riesz energy $E_\alpha(\gamma)$ and
\[E_g(\gamma)=E_\alpha(\gamma-\gamma')=E_\alpha(\gamma)-E_\alpha(\gamma').\]
By representation (\ref{eqq}) of the measure $\mu_{{\mathbf A},\alpha}$ from (ii), we thus have $\mu_{{\mathbf A},\alpha}\in\mathcal E_\alpha(\mathbb R^n)$. In view of (\ref{repr222})  this implies $\dot{\lambda}_{{\mathbf A},\alpha}\in\mathcal E_\alpha({\mathbf A},{\mathbf 1};\mathbb R^n)$ and hence, by
(\ref{dot0}) and Lemma~\ref{l:dot},
\[E_\alpha(\dot{\lambda}_{{\mathbf A},\alpha})=\dot{E}_\alpha(\dot{\lambda}_{{\mathbf A},\alpha})=\dot{w}_\alpha(\mathbf A)=w_\alpha(\mathbf A).\]
This shows that (v) implies (iv) with $\lambda_{{\mathbf A},\alpha}$ given by~(\ref{st-weak-sol}).

We next show that (iv) implies (v) (and hence any of (i)--(iii)). Assume, on the contrary, that $c_\alpha(A_1)=\infty$. Since $c_\alpha(A_2)=\infty$ by the $\alpha$-non-thinness of $A_2$ at infinity (see Remark~\ref{rem-thin}), there exist two sequences $\{\tau_j^i\}_{j\in\mathbb N}$, $i=1,2$, such that $\tau_j^i\in\mathcal E^+_\alpha(A_i,1;\mathbb R^n)$ and
\[\lim_{j\to\infty}\,\|\tau_j^i\|_\alpha=0\text{ \ for all \ }i=1,2.\]
Then $\tau_j:=\tau_j^1-\tau_j^2\in\mathcal E_\alpha({\mathbf A},{\mathbf 1};\mathbb R^n)$ for every $j\in\mathbb N$, and hence
\[\|\tau_j\|^2_\alpha\geqslant w_\alpha(\mathbf A)\geqslant 0.\]
Since $\langle\tau_j^1,\tau_j^2\rangle_\alpha\to0$ as $j\to\infty$ by the Cauchy--Schwarz (Bunyakovski) inequality in $\mathcal E_\alpha(\mathbb R^n)$, we see from the last two displays that $w_\alpha(\mathbf A)=0$. This contradicts (iv) in view of the strict positive definiteness of the kernel $\kappa_\alpha$ because $0\notin\mathcal E_\alpha({\mathbf A},{\mathbf 1};\mathbb R^n)$.

Finally, fix any $\{\mu_k\}_{k\in\mathbb N}\subset\mathcal E_\alpha({\mathbf A},{\mathbf 1};\mathbb R^n)$ with the property (\ref{min-seq}) and also define $\theta_j:=\tilde{\nu}_j-\tilde{\nu}'_j$, where $\{\tilde{\nu}_j-\tilde{\nu}'_j\}_{j\in\mathbb N}\in\dot{\mathbb M}_\alpha({\mathbf A},{\mathbf 1};\mathbb R^n)$ is a (particular) sequence from Theorem~\ref{th-ex}(b$'$). Then $\{\theta_j\}_{j\in\mathbb N}\subset\mathcal E_\alpha({\mathbf A},{\mathbf 1};\mathbb R^n)$ by (\ref{partic}), and it likewise possesses the property (\ref{min-seq}), the latter being clear from (\ref{dot0}) and (\ref{lim:dot}).
Using standard arguments based on the convexity of $\mathcal E_\alpha({\mathbf A},{\mathbf 1};\mathbb R^n)$ and the pre-Hilbert structure on $\mathcal E_\alpha(\mathbb R^n)$, we then obtain
\begin{equation}\label{fundalpha}\lim_{j,k\to\infty}\,\|\mu_k-\theta_j\|_\alpha=0.\end{equation}
Substituting (\ref{st-weak-sol}) into Theorem~\ref{th-ex}(a$'$) with $\{\theta_j\}_{j\in\mathbb N}$ in place of $\{\nu_j-\nu'_j\}_{j\in\mathbb N}$ and next applying (\ref{dot0}) shows that $\theta_j\to\lambda_{{\mathbf A},\alpha}$ in the strong topology of $\mathcal E_\alpha(\mathbb R^n)$. By (\ref{fundalpha}), the same holds for $\{\mu_k\}_{k\in\mathbb N}$. However,  under assumption (\ref{dist}) the class of all $\nu\in\mathfrak M({\mathbf A};\mathbb R^n)\cap\mathcal E_\alpha(\mathbb R^n)$ with $\nu^\pm(\mathbb R^n)\leqslant1$, treated as a metric subspace of $\mathcal E_\alpha(\mathbb R^n)$, is strongly complete and the strong topology on this space is finer than the induced vague topology \cite[Theorem~9.1]{ZPot2}. Thus $\mu_k\to\lambda_{{\mathbf A},\alpha}$ also vaguely.
\end{proof}

\subsection{Remark}\label{rem-st} The requirement that $A_2$ be not  $\alpha$-thin at infinity is essential for the validity of Theorem~\ref{th-st}. Actually, the following theorem holds (see \cite[Theorem~5]{ZR} and Remark~\ref{r-3} above).

\begin{theorem}\label{th-clas}
Let\/ $A_1\subset D$ be a closed set in\/ $\mathbb R^n$ with\/ $c_\alpha(A_1)<\infty$ and\/ ${\rm dist}\,(A_1,A_2)>0$, where $A_2=D^c$. Then Problem\/~{\rm\ref{pr-st}} is solvable if and only if either\/ $c_\alpha(A_2)<\infty$, or\/ $A_2$ is not\/ $\alpha$-thin at infinity.\end{theorem}

Thus, if $A_2$ is $\alpha$-thin at infinity, but $c_\alpha(A_2)=\infty$ (such a set exists according to Remark~\ref{rem-thin}), then Theorem~\ref{th-st}\,(iv) fails to hold in spite of the requirement $c_\alpha(A_1)<\infty$.

\begin{example}Let $n=3$, $\alpha=2$ and $A_2=D^c=F_\varrho$, where $F_\varrho$ is described by (\ref{descr}), and let $A_1$ be a closed set in $\mathbb R^3$ with $c_2(A_1)<\infty$ and ${\rm dist}\,(A_1,A_2)>0$. According to Theorem~\ref{th-clas}, then a solution to Problem~\ref{pr-st} with these data does exist if $\varrho$ is given by either (\ref{c1}) or (\ref{c3}), while the problem has no solution if $\varrho$ is defined by (\ref{c2}). These (theoretical) results on the solvability or unsolvability of Problem~\ref{pr-st} have been illustrated in \cite{HWZ,OWZ} by means of numerical experiments.\end{example}

\section{An open question}\label{open}

As is well known, the concept of standard $\alpha$-Riesz energy coincides with both the weak and the Deny--Schwartz energy when considered on $\mathcal E_\alpha(\mathbb R^n)$.
The concept of weak energy has been shown to coincide with that of Deny--Schwartz energy on $\dot{\mathcal E}_\alpha(\mathbb R^n)$ (see Theorem~\ref{S-D}), while it follows from Theorem~\ref{thm} and the example in \cite[Appendix]{DFHSZ2} that one can choose an element of $\dot{\mathcal E}_\alpha(\mathbb R^n)$ which does not belong to $\mathcal E_\alpha(\mathbb R^n)$. It is however still unknown whether there is a measure $\mu$ with finite Deny--Schwartz energy but infinite weak energy, and this problem makes sense for any $\alpha\in(0,n)$. A similar question for positive measures was raised by Deny \cite[p.~85]{De2}. It is clear however from Theorem~\ref{S-D-comp} that such $\mu$ (provided it exists) must be of noncompact support.

\end{document}